\documentclass{article}
\usepackage{amsfonts,latexsym,hyperref} % label
\newcounter{satznum}
\newtheorem{theorem}{Theorem}[satznum]

\newtheorem{lemma}[theorem]{Lemma}
\newtheorem{corollary}[theorem]{Corollary}

\newtheorem{proposition}[theorem]{Proposition}
\newenvironment{remark}
 {\begin{trivlist}\item[]{\bf Remark.}}
 {\end{trivlist}}
\newenvironment{remarks}
 {\begin{trivlist}\item[]{\bf Remarks.}}
 {\end{trivlist}}

\newenvironment{proof}
 {\begin{trivlist}\item[]{\bf Proof.}}
 {\end{trivlist}}
{\makeatletter
\gdef\cz{{\mathbb C}} % complex numbers
\gdef\me{{\mathbb E}} % expectation
\gdef\nz{{\mathbb N}} % positive integers
\gdef\pr{{\mathbb P}} % probability
\gdef\rz{{\mathbb R}} % real numbers
}
\newcounter{todocounter}

\makeatletter
\def\@MRExtract#1 #2!{#1}
\newcommand{\MR}[1]{% we need to strip the "(...)"
  \xdef\@MRSTRIP{\@MRExtract#1 !}
  \href{http://www.ams.org/mathscinet-getitem?mr=\@MRSTRIP}{MR\@MRSTRIP}}
\makeatother

\begin{document}
   \section*{ON THE BLOCK COUNTING PROCESS AND THE FIXATION LINE OF
   THE BOLTHAUSEN--SZNITMAN COALESCENT}
   {\sc Jonas Kukla and Martin M\"ohle}\footnote{Mathematisches Institut, Eberhard Karls Universit\"at T\"ubingen,
   Auf der Morgenstelle 10, 72076 T\"ubingen, Germany, E-mail addresses:
   jonas.kukla@uni-tuebingen.de, martin.moehle@uni-tuebingen.de}
\begin{center}
%   Draft Version (not ready for submission)\\
   \today\\
%   {\tt c:/texs/jonas/art02.tex}
\end{center}
\begin{abstract}
   The block counting process and the fixation line of the
   Bolthausen--Sznitman coalescent are analyzed.
   Spectral decompositions for their generators and transition probabilities
   are provided leading to explicit expressions for functionals such as
   hitting probabilities and absorption times.
   It is furthermore shown that the block counting process and the
   fixation line of the Bolthausen--Sznitman $n$-coalescent, properly scaled,
   converge in the Skorohod topology %   $D_{[0,\infty)}[0,\infty)$
   to the Mittag--Leffler process and Neveu's continuous-state branching
   process respectively as the sample size $n$ tends to infinity.
   Strong relations to Siegmund duality and to Mehler semigroups and
   self-decomposability are pointed out.

   \vspace{2mm}

   \noindent Keywords:
   block counting process; Bolthausen--Sznitman coalescent;
   fixation line; Mehler semigroup; Mittag--Leffler process;
   Neveu's continuous-state branching process;
   self-decomposability; Siegmund duality; spectral decomposition

   \vspace{2mm}

   \noindent Running head: On the Bolthausen--Sznitman coalescent

   \vspace{2mm}

   \noindent 2010 Mathematics Subject Classification:
            Primary 60F05;   % Central limit and other weak theorems
            60J27            % Continuous-time Markov processes on discrete
                             % state space
            Secondary 92D15; % Problems related to evolution
            97K60            % Distributions and stochastic processes
\end{abstract}
\subsection{Introduction} \label{intro}
\setcounter{theorem}{0}
   Exchangeable coalescents are Markovian processes $(\Pi_t)_{t\ge 0}$
   with state space ${\cal P}$, the set of partitions of $\nz:=\{1,2,\ldots\}$.
   These processes have attracted the interest of several
   researchers, mainly in biology, mathematics and physics, during the last decades.
   The full family of exchangeable coalescents (with simultaneous multiple collisions)
   is a class of partition valued Markovian processes with a rich
   probabilistic structure and hence important
   for mathematical studies. Moreover, coalescents are useful in mathematical population
   genetics to model the ancestry of a sample of individuals or genes and
   therefore important for biological applications.

   Exchangeable coalescents with multiple collisions
   but without simultaneous multiple collisions
   are characterized by a measure $\Lambda$ on the unit interval $[0,1]$ and
   therefore called $\Lambda$-coalescents.
   For further information on these processes we refer the reader
   to the independent works of Pitman \cite{pitman} and
   Sagitov \cite{sagitov2}. The most important coalescent is probably the
   Kingman coalescent \cite{kingman}, which allows only for binary mergers of ancestral
   lineages. In this case the measure $\Lambda$ is the Dirac measure
   at $0$.

   For $t\ge 0$ let $N_t$ denote the number of blocks of
   $\Pi_t$ and let $N_t^{(n)}$ denote the number of blocks of $\Pi_t^{(n)}$,
   where $\Pi_t^{(n)}$ denotes the partition of $\Pi_t$ restricted to a
   sample of size $n\in\nz$.
   The processes $(N_t)_{t\ge 0}$ and $(N_t^{(n)})_{t\ge 0}$
   are called the block counting processes of $(\Pi_t)_{t\ge 0}$ and
   $(\Pi_t^{(n)})_{t\ge 0}$ respectively.

   H\'enard \cite{henard2} introduced the so-called
   fixation line $(L_t)_{t\ge 0}$ of a $\Lambda$-coalescent.
   Recently \cite{gaisermoehle} the fixation line was extended to arbitrary
   exchangeable coalescents. One possible definition of the fixation line is based on the
   lookdown construction going back to Donnelly and Kurtz
   \cite{donnellykurtz1,donnellykurtz2}. For further information on the
   fixation line we refer the reader to \cite{gaisermoehle} and \cite{henard2}.
   In the following $(L_t^{(n)})_{t\ge 0}$ denotes the fixation line
   with initial state $L_0^{(n)}=n$.

   In this article we focus on the Bolthausen--Sznitman coalescent
   \cite{bolthausensznitman}, which is the particular $\Lambda$-coalescent
   with $\Lambda$ being the uniform distribution on the unit interval.
   The generator $Q=(q_{ij})_{i,j\in\nz}$ of the block counting process
   and the generator
   $\Gamma=(\gamma_{ij})_{i,j\in\nz}$ of the fixation line
   of the Bolthausen--Sznitman coalescent have entries
   (see, for example, \cite[Eq.~(1.1)]{moehlepitters} and
   \cite[p.~3015, Eq.~(2.8) with $\alpha=1$]{henard2})
   $$
   q_{ij}\ =\
   \left\{
      \begin{array}{cl}
         \displaystyle\frac{i}{(i-j)(i-j+1)} & \mbox{for $j<i$}\\
         1-i & \mbox{for $j=i$},\\
         0 & \mbox{for $j>i$.}
      \end{array}
   \right.
   \quad\mbox{and}\quad
   \gamma_{ij}\ =\
   \left\{
      \begin{array}{cl}
      \displaystyle\frac{i}{(j-i)(j-i+1)} & \mbox{for $j>i$,}\\
      -i & \mbox{for $j=i$,}\\
      0 & \mbox{for $j<i$.}
      \end{array}
   \right.
   $$
   The block counting process and the corresponding generator $Q$ have
   been studied intensively in the literature. In this article we focus on
   both processes $(N_t)_{t\ge 0}$ and $(L_t)_{t\ge 0}$
   with an emphasis on the fixation line $(L_t)_{t\ge 0}$,
   which has been studied less intensively so far. We furthermore
   stress the duality relation between both processes.

   In Section \ref{spectral} the processes $(N_t)_{t\ge 0}$
   and $(L_t)_{t\ge 0}$ are
   analyzed with an emphasis on spectral decompositions. These
   spectral decompositions lead to explicit expressions for several
   functionals of these processes such as hitting probabilities and
   absorption times.

   Section \ref{asymptotics} deals with the behavior of the block
   counting process $(N_t^{(n)})_{t\ge 0}$ and the fixation line
   $(L_t^{(n)})_{t\ge 0}$ as the sample size $n$ tends to infinity.
   The main convergence result (Theorem \ref{main}) states that both
   processes, properly scaled, converge in the Skorohod sense as
   $n\to\infty$ to the Mittag--Leffler process and to Neveu's
   continuous-state branching process respectively.

   The proofs provided in Section \ref{proofs} rely on both analytic
   and probabilistic arguments which demonstrates the interplay between
   analysis and probability. A short appendix collects some results of
   independent interest used in the proofs.

\subsection{Spectral decompositions and applications} \label{spectral}
\setcounter{theorem}{0}
Spectral decompositions are of fundamental interest since
they lead to diagonal representations of the corresponding operators
or matrices which simplify many mathematical calculations and
numerical computations significantly. Explicit spectral decompositions
for (the block counting process of) the Kingman coalescent and the
Bolthausen--Sznitman coalescent are provided in \cite{kuklapitters}
and \cite{moehlepitters}. We are interested in analog spectral decompositions
for the fixation line.
%Suppose that $\Xi(\Delta\setminus\Delta^*)>0$. Then
%the eigenvalues $\gamma_i:=-\gamma_{ii}$, $i\in\nz$, are pairwise distinct.
%It follows that the generator matrix $\Gamma:=(\gamma_{ij})_{i,j\in\nz}$ has
%a spectral decomposition $\Gamma=RDL$.
A spectral decomposition of the generator $\Gamma$ of the fixation
line of the Kingman coalescent is provided in the appendix (Lemma
\ref{kingmanspectral}) for completeness.
Our first result (Theorem \ref{fixationtheo}) provides an explicit spectral decomposition for the
generator $\Gamma$ of the fixation line of the Bolthausen--Sznitman
coalescent. In the following $s(i,j)$ and $S(i,j)$ denote the
Stirling numbers of the first and second kind respectively.
\begin{theorem}[Spectral decomposition of the generator of the fixation line]
\label{fixationtheo}
\ \\
   The generator $\Gamma=(\gamma_{ij})_{i,j\in\nz}$ of the fixation line
   $(L_t)_{t\ge 0}$ of the Bolthausen--Sznitman coalescent has
   spectral decomposition $\Gamma=RDL$, where $D=(d_{ij})_{i,j\in\nz}$
   is the diagonal matrix with entries $d_{ij}=-i$ for $i=j$ and $d_{ij}=0$
   for $i\ne j$, and $R=(r_{ij})_{i,j\in\nz}$ and $L=(l_{ij})_{i,j\in\nz}$
   are upper right triangular matrices with entries
   \begin{equation} \label{randl}
      r_{ij}\ =\ \frac{i!}{j!}(-1)^{i+j}S(j,i)
      \quad\mbox{and}\quad
      l_{ij}\ =\ \frac{i!}{j!}(-1)^{i+j}s(j,i),
      \qquad i,j\in\nz.
   \end{equation}
\end{theorem}
The following corollaries demonstrate that spectral decompositions are
useful to analyze the underlying processes.
\begin{corollary}[Branching property/transition probabilities of the fixation line] \label{trans}
   \ \\
   For the Bolthausen--Sznitman coalescent, the random variable $L_t^{(i)}$
   has probability generating function (pgf)
   \begin{equation} \label{trans1}
      \me(z^{L_t^{(i)}})\ =\ (1-(1-z)^{e^{-t}})^i,
      \qquad |z|<1, t\ge 0,i\in\nz.
   \end{equation}
   Thus, $(L_t)_{t\ge 0}$ is a Markovian continuous-time branching
   process with state space $\nz$ and offspring distribution
   $p_k=1/(k(k-1))$, $k\in\{2,3,\ldots\}$ having infinite mean.
   Moreover, the transition probabilities $p_{ij}(t):=\pr(L_t=j\,|\,L_0=i)
   =\pr(L_t^{(i)}=j)$ are given by
   \begin{equation} \label{trans2}
      p_{ij}(t)
      \ = \ (-1)^{i+j}\frac{i!}{j!}\sum_{k=i}^j S(k,i) e^{-tk} s(j,k)
      \ = \ (-1)^j\sum_{k=1}^i (-1)^k {i\choose k}{{e^{-t}k}\choose j},
   \qquad i,j\in\nz.
   \end{equation}
\end{corollary}
\begin{remarks}
\begin{enumerate}
   \item[1.] For $i=1$ it follows that $L_t=L_t^{(1)}$ has pgf
      $\me(z^{L_t})=1-(1-z)^\alpha=-\sum_{j=1}^\infty {\alpha\choose j} (-z)^j$
      and distribution
      \begin{equation} \label{alphadist}
         \pr(L_t=j)
         \ =\ p_{1j}(t)
         \ =\ (-1)^{j+1}{\alpha\choose j}
         \ =\ \frac{\alpha\Gamma(j-\alpha)}{\Gamma(1-\alpha)\Gamma(j+1)},
      \qquad j\in\nz,
      \end{equation}
      where $\alpha:=e^{-t}$. Note that $\pr(L_t=j)\sim\alpha/(\Gamma(1-\alpha)
      j^{\alpha+1})$ as $j\to\infty$ and that $L_t$ has a Pareto like tail
      $\pr(L_t\ge j)=\Gamma(j-\alpha)/(\Gamma(1-\alpha)\Gamma(j))\sim
      1/(\Gamma(1-\alpha)j^\alpha)$ as $j\to\infty$.
      Thus, $\me(L_t^q)=\sum_{j=1}^\infty
      j^q\pr(L_t=j)<\infty$ if and only if $q<\alpha$. Particular
      reciprocal factorial moments of $L_t$ are known explicitly. For
      example,
      $$
      \me\bigg(\frac{1}{(L_t+1)(L_t+2)\cdots (L_t+k)}\bigg)
      \ =\ \frac{\alpha}{\Gamma(1-\alpha)}\sum_{j=1}^\infty
           \frac{\Gamma(j-\alpha)}{\Gamma(j+k+1)}
      \ =\ \frac{\alpha}{k!(\alpha+k)},\ \ k\in\nz.
      $$
      The distribution (\ref{alphadist}) and similar distributions occur in
      \cite[p.~9]{huilletmoehle}, \cite[p.~225]{iksanovmoehle} and
      \cite[p.~70, Eq.~(3.38)]{pitman2}.
   \item[2.] The pgf $f(s):=\sum_{k=2}^\infty p_ks^k=s+(1-s)\log(1-s)$ of the
      offspring distribution satisfies
      $$
      \int_{(1-\varepsilon,1)} \frac{\lambda({\rm d}s)}{f(s)-s}
      \ =\ \int_{(1-\varepsilon,1)}\frac{\lambda({\rm d}s)}{(1-s)\log(1-s)}
      \ =\ \int_{(0,\varepsilon)}\frac{\lambda({\rm d}x)}{x\log x}
      \ =\ -\infty
      $$
      for all $\varepsilon\in (0,1)$, where $\lambda$ denotes Lebesgue
      measure on $(0,1)$. This implies (Harris \cite[p.~107]{harris})
      that the fixation line $(L_t)_{t\ge 0}$ does not explode, in agreement
      (see \cite{gaisermoehle}) with the fact that the Bolthausen--Sznitman
      coalescent stays infinite.
\end{enumerate}
\end{remarks}
As a second application we study the probability
$h(i,j)=\pr(L_t^{(i)}=j\mbox{ for some $t\ge 0$})$
that the fixation line hits state $j\in\nz$ started from state $i\in\nz$.
\begin{corollary}[Hitting probabilities] \label{hit}
   The hitting probabilities $h(i,j)$ have horizontal generating function
   \begin{equation} \label{hitgen}
      \sum_{j=i}^\infty h(i,j)z^{j-1}
      \ =\ \frac{z^i}{(1-z)(-\log(1-z))},\qquad i\in\nz, |z|<1.
   \end{equation}
   In particular $h(i,j)=h(1,j-i+1)$ depends on $i$ and $j$ only via $j-i$.
   Moreover, for all $i\in\nz$,
   \begin{equation} \label{hitasy}
      h(i,j)\ =\ \frac{1}{\log j} - \frac{\gamma}{\log^2j} +
      O\bigg(\frac{1}{\log^3j}\bigg),\qquad j\to\infty,
   \end{equation}
   where $\gamma:=-\Gamma'(1)\approx 0.577216$ denotes the Euler--Mascheroni
   constant. The hitting probability $h(i,j)$ can be computed via
   \begin{equation} \label{hit0}
      h(i,j)\ =\ \sum_{k=1}^{j-i}\pr(\eta_1+\cdots+\eta_k=j-i),
      \qquad 1\le i<j,
   \end{equation}
   where $\eta_1,\eta_2,\ldots$ are iid random variables with distribution
   $\pr(\eta_1=n):=u_n:=1/(n(n+1))$, $n\in\nz$.
   The hitting probabilities can be also expressed in terms of
   the Stirling numbers $s(.,.)$ and $S(.,.)$ of the first and second kind
   as
   \begin{eqnarray}
      h(i,j)
      & = & (-1)^{i+j}\frac{i!}{(j-1)!}\sum_{k=i}^j\frac{s(j,k)S(k,i)}{k}\label{hit1a}\\
      & = & (-1)^{j-i}\frac{1}{(j-i)!}\sum_{k=1}^{j-i+1}\frac{s(j-i+1,k)}{k},\label{hit1b}
      \qquad 1\le i\le j.
   \end{eqnarray}
   Moreover, $h(i,j)$ has representations
   \begin{equation} \label{hit2}
      h(i,j)
      \ =\ \frac{1}{(j-i)!}\int_0^1 \frac{\Gamma(j-i+x)}{\Gamma(x)}\,{\rm d}x
      \ =\ \frac{1}{(j-i)!}\sum_{k=0}^{j-i}\frac{|s(j-i,k)|}{k+1},
      \qquad 1\le i\le j.
   \end{equation}
\end{corollary}
\begin{remark}
   Concrete values of the hitting probabilities $h(i,j)$ for $i=1$ and
   $j\in\{1,\ldots,7\}$ are provided in the remark after the proof of
   Corollary \ref{hit}.
\end{remark}
We now turn to the block counting process $(N_t^{(n)})_{t\ge 0}$ of the
Bolthausen--Sznitman $n$-coalescent. For $n\in\nz$ and
$i\in\{1,\ldots,n\}$ let $\tau_{ni}:=\inf\{t>0\,:\,N_t^{(n)}\le i\}$
denote the first time the block counting process
$(N_t^{(n)})_{t\ge 0}$ jumps to a state smaller
than or equal to $i$. Note that $\tau_n:=\tau_{n1}$ is the
absorption time of $N^{(n)}$.
\begin{corollary}[Distribution function and asymptotics of ${\mathbf\tau_{ni}}$]
   \label{dist}
   For all $n\in\nz$ and $i\in\{1,\ldots,n\}$, $\tau_{ni}$ has
   distribution function
   \begin{equation} \label{distoftau}
      \pr(\tau_{ni}\le t)
      \ =\ \sum_{j=1}^i (-1)^{n+j}{i\choose j}
      {{je^{-t}-1}\choose {n-1}},
      \qquad t\in (0,\infty).
   \end{equation}
   In particular, for every $i\in\nz$,
   $\tau_{ni}-\log\log n\ \to\ \min(G_1,\ldots,G_i)$
   in distribution as $n\to\infty$, where $G_1,G_2,\ldots$ are
   independent standard Gumbel distributed random variables.
\end{corollary}
\begin{remark}
   Note that $\min(G_1,\ldots,G_i)$ has distribution function
   $F_i(x):=1-(1-F(x))^i$, where $F(x):=e^{-e^{-x}}$, $x\in\rz$. For $i=1$
   we recover the well known convergence result (see Goldschmidt and Martin
   \cite[Proposition 3.4]{goldschmidtmartin}, Freund and M\"ohle \cite[Corollary 1.2]{freundmoehle} or
   H\'enard \cite[Theorem 3.9]{henard2}) that the scaled absorption time
   $\tau_n-\log\log n$ is asymptotically
   standard Gumbel distributed.
\end{remark}
The fact that the distribution function (\ref{distoftau}) of $\tau_{ni}$ is
known explicitly can be further exploited. For example, the following
Edgeworth expansion holds.
\begin{corollary}[Edgeworth expansion] \label{edgeworth}
   For every $i\in\nz$ and $x\in\rz$ the following Edgeworth expansion of order
   $K\in\nz_0$ holds.
   \begin{equation} \label{edgeworth1}
      \pr(\tau_{ni}-\log\log n\le x)
      \ =\ \sum_{k=0}^K c_k d_{ki}(x)\frac{e^{-kx}}{\log^k n}
      + O\bigg(\frac{1}{\log^{K+1}n}\bigg),\qquad n\to\infty,
   \end{equation}
   where $c_0,c_1,\ldots$ are the coefficients
   in the series expansion $1/\Gamma(1-x)=\sum_{k=0}^\infty c_kx^k$,
   $|x|<1$, and
   \begin{equation} \label{dki1}
      d_{ki}(x)
      \ := \ \bigg(e^x\frac{{\rm d}}{{\rm d}x}\bigg)^k F_i(x)
      \ = \ \sum_{j=1}^i (F(x))^j (-1)^{j-1}{i\choose j} j^k,
      \qquad k\in\nz_0, i\in\nz, x\in\rz,
   \end{equation}
   with $F_i$ and $F$ as defined in the previous remark.
   Alternatively, $d_{0i}(x)=F_i(x)$ and
   \begin{equation} \label{dki2}
      d_{ki}(x)\ =\
      \sum_{j=1}^k S(k,j) (-1)^{j-1}(i)_j(F(x))^j(1-F(x))^{i-j},
      \qquad k,i\in\nz, x\in\rz,
   \end{equation}
   where the $S(k,j)$ are the Stirling numbers of the second kind and
   $(i)_j:=i(i-1)\cdots(i-j+1)$.
\end{corollary}
\begin{remarks}
   \begin{enumerate}
      \item[1.] The coefficients $c_k$, $k\in\nz_0$, are related to the
         moments of the Gumbel distribution (see Lemma \ref{coef}). The
         concrete values $c_k$ for $k\le 3$ are provided in the remark
         after the proof of Lemma \ref{coef}.
      \item[2.] For $K=1$ Corollary \ref{edgeworth} reads
         $\pr(\tau_{ni}-\log\log n\le x)=F_i(x)-\gamma F_i'(x)/\log n
         + O(1/\log^2n)$. In particular, for every $x\in\rz$,
         $\pr(\tau_{ni}-\log\log n\le x)-F_i(x)\sim-\gamma F_i'(x)/\log n$
         as $n\to\infty$. Thus, the speed of the convergence of
         $\tau_{ni}-\log\log n$ to $G_i$ is of order $1/\log n$.
   \end{enumerate}
\end{remarks}
\subsection{Asymptotics} \label{asymptotics}
\setcounter{theorem}{0}
We are interested in the behavior of the block counting process
$(N_t^{(n)})_{t\ge 0}$ and the fixation line
$(L_t^{(n)})_{t\ge 0}$ of the Bolthausen--Sznitman
$n$-coalescent as the sample size $n$ tends to infinity. In order to
state the main convergence result (see Theorem \ref{main} below) let us recall
some properties of the Mittag--Leffler process
$X=(X_t)_{t\ge 0}$ and Neveu's \cite{neveu} continuous-state branching
process $Y=(Y_t)_{t\ge 0}$.

The Mittag--Leffler process $X$ is a Markovian process in continuous time
with state space $E:=[0,\infty)$. The name of this process comes from the
fact that for every $t\ge 0$ the marginal random variable $X_t$
is Mittag--Leffler distributed with parameter $e^{-t}$. Note that $X_t$
has moments $\me(X_t^m)=\Gamma(1+m)/\Gamma(1+me^{-t})$, $m\in [0,\infty)$.
The semigroup $(T_t^X)_{t\ge 0}$ of the Mittag--Leffler process $X$ is given
by
\begin{equation} \label{semix}
   T_t^Xf(x)\ =\ \me(f(x^{e^{-t}}X_t)),\qquad t,x\ge 0, f\in B(E),
\end{equation}
where $B(E)$ denotes the set of all bounded measurable functions
$f:E\to\rz$. Some further information on the process $X$ can be found
in \cite{baurbertoin} and \cite{moehlemittag}.

Neveu's \cite{neveu} continuous-state branching process $Y$ is as well a
Markovian process in continuous time with state space $E$. For every
$t\ge 0$ the marginal random variable $Y_t$ is $\alpha$-stable with
Laplace transform $\me(e^{-\lambda Y_t})=e^{-\lambda^\alpha}$,
$\lambda\ge 0$, where $\alpha:=e^{-t}$. The semigroup $(T_t^Y)_{t\ge 0}$
of Neveu's continuous-state branching process $Y$ is given by
\begin{equation} \label{semiy}
   T_t^Yg(y)\ =\ \me(g(y^{e^t}Y_t)),\qquad t,y\ge 0, g\in B(E).
\end{equation}
Note that (see, for example, \cite{moehlemittag}) the Mittag--Leffler
process $X$ is Siegmund dual to Neveu's continuous state branching
process $Y$, i.e. $\pr(X_t\le y\,|\,X_0=x)=\pr(Y_t\ge x\,|\,Y_0=y)$
for all $t,x,y\ge 0$.

Define the
scaled block counting process $X^{(n)}=(X_t^{(n)})_{t\ge 0}$ and the
scaled fixation line $Y^{(n)}:=(Y_t^{(n)})_{t\ge 0}$ of the
Bolthausen--Sznitman $n$-coalescent via
\begin{equation} \label{xy}
   X_t^{(n)}\ :=\ \frac{N_t^{(n)}}{n^{e^{-t}}}
   \quad\mbox{and}\quad
   Y_t^{(n)}\ :=\ \frac{L_t^{(n)}}{n^{e^t}},
   \qquad t\ge 0, n\in\nz.
\end{equation}
Note that, for $n\ge 2$, the processes $X^{(n)}$ and $Y^{(n)}$ are
time-inhomogeneous because of the time-dependent scalings $n^{e^{-t}}$
and $n^{e^t}$. We are now able to state the main convergence result.
The proof of the following theorem is provided in Section \ref{proofs}.
\begin{theorem}[Asymptotics of the block counting process and the fixation line]% $\mathbf N$ and $\mathbf L$]
\label{main}
\ \\
   For the Bolthausen--Sznit\-man coalescent the following two assertions hold.
   \begin{enumerate}
      \item[a)] As $n\to\infty$ the scaled block counting process
         $X^{(n)}$, defined in (\ref{xy}), converges in
         $D_E[0,\infty)$ to the Mittag--Leffler process $X=(X_t)_{t\ge 0}$.
      \item[b)] As $n\to\infty$ the scaled fixation line
         $Y^{(n)}$, defined in (\ref{xy}), converges in
         $D_E[0,\infty)$ to Neveu's continuous-state
         branching process $Y=(Y_t)_{t\ge 0}$.
   \end{enumerate}
\end{theorem}
   Theorem \ref{main} demonstrates the intimate relation between the
   Bolthausen--Sznitman coalescent, the Mittag--Leffler process
   and Neveu's continuous state branching process. We refer the reader
   to Bertoin and Le Gall \cite{bertoinlegall} for further insights
   concerning these relations.

   Theorem \ref{main} a) is known from the literature
   \cite[Theorem 1.1]{moehlemittag} and provided here for completeness.
   Our proof of Theorem \ref{main} a) is significantly shorter than
   the proof provided in \cite{moehlemittag} and gives further
   insights into the structure of the scaled block counting process
   $X^{(n)}$.

   Part b) of Theorem \ref{main} is likely to be known from
   branching process theory, however the authors have not been able
   to trace this result in the literature. Note that the offspring
   distribution of the branching process $(L_t^{(n)})_{t\ge 0}$ has pgf
   $f(s)=s+(1-s)\log(1-s)$ and, hence, infinite mean. For related
   convergence results for the critical case when the offspring
   distribution has mean $1$ we refer the reader to
   Sagitov \cite{sagitov1} and the references therein.
   Note that in Theorem 2.1 of \cite{sagitov1} the space-scaling is $n$ and
   an additional time-scaling occurs. Theorem \ref{main} b)
   may be viewed as a kind of boundary case of Theorem 2.1 of
   \cite{sagitov1} for $\alpha\to 1$.
   Similar convergence results for sequences of discrete-time
   branching processes can be traced back to Lamperti \cite{lamperti1,lamperti2}.

\vspace{2mm}

In summary the following commutative diagram holds.
$$
\begin{array}{ccc}
\fbox{
      \begin{tabular}{c}
         Scaled block counting process\\
         $(N_t^{(n)}/n^{e^{-t}})_{t\ge 0}$
      \end{tabular}
}
      & \Rightarrow &
\fbox{
   \begin{tabular}{c}
      Mittag--Leffler process\\
      $(X_t)_{t\ge 0}$
   \end{tabular}
}\\
   & & \\
   \updownarrow & & \updownarrow\\
   & & \\
\fbox{
   \begin{tabular}{c}
      Scaled fixation line process\\
      $(L_t^{(n)}/n^{e^t})_{t\ge 0}$
   \end{tabular}
}
      & \Rightarrow &
\fbox{
   \begin{tabular}{c}
      Neveu's branching process\\
      $(Y_t)_{t\ge 0}$
   \end{tabular}
}
\end{array}
$$
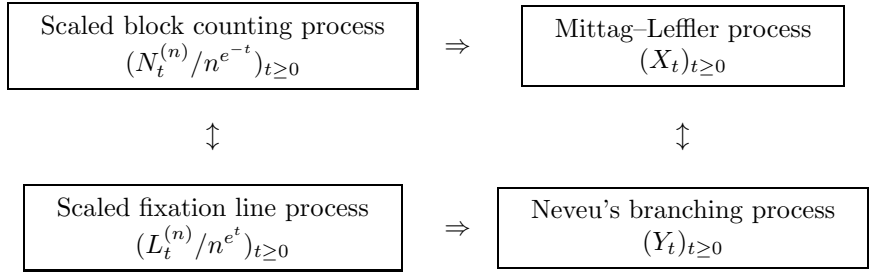
Figure 1: Commutative diagram for the block counting process
$(N_t^{(n)})_{t\ge 0}$ and the fixation line $(L_t^{(n)})_{t\ge 0}$ of the
Bolthausen--Sznitman coalescent. The right-arrows `$\Rightarrow$' stand for
`convergence in $D_E[0,\infty)$ as $n\to\infty$'. The vertical
updown-arrows `$\updownarrow$' stand for `duality', on the left hand side the duality of the
block counting process $(N_t)_{t\ge 0}$ and the fixation line
$(L_t)_{t\ge 0}$ with respect to the Siegmund duality kernel $H:\nz^2\to\{0,1\}$
defined via $H(i,j):=1$ for $i\le j$ and $H(i,j):=0$ otherwise, on the
right hand side the duality of $(X_t)_{t\ge 0}$
and $(Y_t)_{t\ge 0}$
with respect to the Siegmund duality kernel $H:[0,\infty)^2\to\{0,1\}$ defined via
$H(x,y):=1$ for $x\le y$ and $H(x,y):=0$ otherwise.

\vspace{5mm}

We finally point out that
Theorem 3.1 is strongly related to Mehler semigroups, to self-decomposability
and to the Gumbel distribution. Clearly, Theorem \ref{main} can be stated
logarithmically as follows. The process $(\log N_t^{(n)}-e^{-t}\log n)_{t\ge 0}$ converges
in $D_\rz[0,\infty)$ to $\tilde{X}:=(\tilde{X}_t)_{t\ge 0} :=(\log X_t)_{t\ge 0}$
and the process $(\log L_t^{(n)}-e^t\log n)_{t\ge 0}$ converges in
$D_\rz[0,\infty)$ to $\tilde{Y}:=(\tilde{Y}_t)_{t\ge 0}:=(\log Y_t)_{t\ge 0}$
as $n\to\infty$. Note that the semigroup $(T_t^{\tilde{X}})_{t\ge 0}$ of
$\tilde{X}$ is given by
\begin{equation} \label{semigrouplogx}
   T_t^{\tilde{X}}f(x)\ =\ \me(f(xe^{-t}+\tilde{X}_t)),
   \qquad t\ge 0, f\in B(\rz), x\in\rz,
\end{equation}
whereas the semigroup $(T_t^{\tilde{Y}})_{t\ge 0}$ of $\tilde{Y}$ is given by
\begin{equation} \label{semigrouplogy}
   T_t^{\tilde{Y}}g(y)\ =\ \me(g(ye^t+\tilde{Y}_t)),
   \qquad t\ge 0, g\in B(\rz), y\in\rz.
\end{equation}
Semigroups of this form belong to the class of so called Mehler semigroups.
Note that (\ref{semigrouplogx}) and (\ref{semigrouplogy}) define the
semigroups of $\tilde{X}$ and $\tilde{Y}$ completely, since
for every $t\ge 0$ the distributions of the marginals
$\tilde{X}_t=\log X_t$ and $\tilde{Y}_t=\log Y_t$
can be characterized as follows. Let $E$ be standard exponentially
distributed and independent of $X$ and $Y$. Note that $G:=-\log E$ is
standard Gumbel distributed.
From $E\stackrel{d}{=}(E/Y_t)^{e^{-t}}$ (see, for example,
\cite{shanbhagsreehari}) we conclude by an application of the
transformation $x\mapsto -\log x$ that the distribution of
$\tilde{Y}_t$ is characterized via the self-decomposable
distributional equation
$$
G\ \stackrel{d}{=}\ e^{-t}G+e^{-t}\tilde{Y}_t.
$$
Thus, $\tilde{Y}_t$ has characteristic function
$u\mapsto\Gamma(1-iue^t)/\Gamma(1-iu)$, $u\in\rz$, and cumulants
$\kappa_j(\tilde{Y}_t)=(e^{jt}-1)\kappa_j(G)$, $j\in\nz$, $t\ge 0$,
where $\kappa_j(G)$ are the cumulants of the Gumbel distribution,
i.e. $\kappa_1(G)=\gamma$ (Euler--Mascheroni constant) and $\kappa_j(G)=
(-1)^j\Psi^{(j-1)}(1)=(j-1)!\zeta(j)$ for $j\in\nz\setminus\{1\}$,
where $\Psi$ and $\zeta$ denote the digamma function (logarithmic
derivative of the gamma function) and the Riemann zeta function
respectively.

Similarly, the distribution of $\tilde{X}_t$ is characterized via the
self-decomposable distributional equation
$$
S\ \stackrel{d}{=}\ e^{-t}S+\tilde{X}_t,
$$
where $S:=-G$. Therefore, $\tilde{X}_t$ has characteristic function
$u\mapsto \Gamma(1+iu)/\Gamma(1+iue^{-t})$, $u\in\rz$, and
cumulants $\kappa_j(\tilde{X}_t)=(-1)^j(1-e^{-jt})\kappa_j(G)$,
$j\in\nz$, $t\ge 0$.

\subsection{Proofs} \label{proofs}
\setcounter{theorem}{0}
\begin{proof} (of Theorem \ref{fixationtheo})
   Two proofs are provided. The first proof is self-contained and based
   on generating functions. The second proof uses duality and the spectral
   decomposition \cite[Theorem 1.1]{moehlepitters} of the generator of
   the block counting process.

   \vspace{2mm}

   {\bf Proof 1.} (via generating functions)

   The proof is similar to that of Theorem 1.1 of \cite{moehlepitters}.
   Let $D=(d_{ij})_{i,j\in\nz}$ be the diagonal matrix with entries
   $d_{ii}:=-\gamma_i=\gamma_{ii}$, $i\in\nz$. Furthermore, let $R=(r_{ij})_{i,j\in\nz}$
   be the upper right triangular matrix with entries defined for each
   $j\in\nz$ recursively via $r_{jj}:=1$ and
   \begin{equation} \label{rrec}
      r_{ij}\ :=\ \frac{1}{\gamma_i-\gamma_j}\sum_{k=i+1}^j
      \gamma_{ik}r_{kj},\qquad i\in\{j-1,j-2,\ldots,1\}.
   \end{equation}
   Since $\gamma_{ii}=-\gamma_i$, $i\in\nz$, we conclude that
   $r_{ij}\gamma_{jj}=\sum_{k=i}^j \gamma_{ik}r_{kj}$. Thus, the
   entries of $R$ are defined such that $RD=\Gamma R$. Define
   $L:=R^{-1}$. Then, the spectral decomposition $\Gamma=RDL$ holds.
   Moreover, $DL=L\Gamma$ and, hence, $\gamma_{ii}l_{ij}=\sum_{k=i}^j l_{ik}\gamma_{kj}$,
   $i,j\in\nz$. Since $\gamma_{ii}=-\gamma_i$, $i\in\nz$, we obtain for each
   $i\in\nz$ the recursion $l_{ii}=1$ and
   \begin{equation} \label{lrec}
      l_{ij}\ =\ \frac{1}{\gamma_j-\gamma_i}\sum_{k=i}^{j-1}l_{ik}\gamma_{kj},
      \qquad j\in\{i+1,i+2,\ldots\}.
   \end{equation}
   Let $U:=\{z\in\cz:|z|<1\}$ denote the open unit disc.
   For $i\in\nz$ define the generating function $l_i:U\to\cz$
   via $l_i(z):=\sum_{j=i}^\infty l_{ij}z^j$, $z\in U$, and consider
   the modified function $f_i:U\to\cz$ defined via
   $f_i(z):=\sum_{j=i}^\infty (j-i)l_{ij}z^j$, $z\in U$. We have
   $$
   f_i(z)\ =\ \sum_{j=i}^\infty jl_{ij}z^j - i\sum_{j=i}^\infty l_{ij}z^j
   \ =\ zl_i'(z)-il_i(z).
   $$
   On the other hand, by the recursion (\ref{lrec}), we obtain the
   factorization
   \begin{eqnarray*}
      f_i(z)
      & = & \sum_{j=i+1}^\infty (j-i)l_{ij}z^j
      \ = \ \sum_{j=i+1}^\infty \sum_{k=i}^{j-1} l_{ik}\gamma_{kj}z^j\\
      & = & \sum_{k=i}^\infty l_{ik}\sum_{j=k+1}^\infty \gamma_{kj}z^j
      \ = \ \sum_{k=i}^\infty kl_{ik}z^k \sum_{j=k+1}^\infty
            \frac{z^{j-k}}{(j-k)(j-k+1)}\\
      & = & \sum_{k=i}^\infty kl_{ik}z^k \sum_{n=1}^\infty \frac{z^n}{n(n+1)}
      \ = \ zl_i'(z) a(z),
   \end{eqnarray*}
   where the auxiliary function $a:U\to\cz$ is defined via
   $a(z):=\sum_{n=1}^\infty z^n/(n(n+1))=1-(1-z)(-\log(1-z))/z$,
   $z\in U$. Thus, $l_i$ satisfies the differential equation
   $zl_i'(z)a(z)=zl_i'(z)-il_i(z)$ or, equivalently,
   $$
   l_i'(z)\ =\ \frac{il_i(z)}{(1-a(z))z}
   \ =\ \frac{il_i(z)}{(1-z)(-\log(1-z))}.
   $$
   The solution of this homogeneous differential equation with initial
   conditions $l_i(0)=\cdots=l_i^{(i-1)}(0)=0$ and $l_i^{(i)}(0)=i!$ is
   $l_i(z)=(-\log(1-z))^i$, $i\in\nz$, $z\in U$.
   Here $l_i^{(j)}$ denotes the $j$th derivative of $l_i$.
   For $f(z)=\sum_{j=0}^\infty a_jz^j$ let $[z^j]f(z):=a_j$ denote the
   coefficient in front of $z^j$ in the series expansion of $f$.
   By \cite[p.~824]{abramowitzstegun},
   $l_i(z)=(-\log(1-z))^i=i!\sum_{j=i}^\infty |s(j,i)|z^j/j!$
   and, hence,
   $$
   l_{ij}\ =\ [z^j]l_i(z)
   \ =\ \frac{i!}{j!}|s(j,i)|
   \ =\ \frac{i!}{j!}(-1)^{i+j}s(j,i),
   $$
   which is the second formula in (\ref{randl}).
   Let us now turn to the inverse $R=L^{-1}$ of $L$. We have
   $L(z,z^2,\ldots)^\top=(l_1(z),l_2(z),\ldots)^\top$. Multiplying from
   the left with $R$ it follows that $(z,z^2,\ldots)^\top=R(l_1(z),l_2(z),\ldots)^\top$.
   Thus, $z^i=\sum_{j=i}^\infty r_{ij}l_j(z)=\sum_{j=i}^\infty r_{ij}(-\log(1-z))^j$. Replacing $z$ by
   $1-e^{-z}$ leads to
   $(1-e^{-z})^i=\sum_{j=i}^\infty r_{ij}z^j=:r_i(z)$, $i\in\nz$, $z\in U$.
   The calculations between Eq. (2.9) and Eq. (2.10) in \cite{moehlepitters}
   show that $r_i$ has expansion
   $$
   r_i(z)
   \ =\ (1-e^{-z})^i
   \ =\ \sum_{j=0}^\infty (-1)^{i+j} \frac{i!}{j!}S(j,i)z^j,
   $$
   which yields the formula in (\ref{randl}) for the coefficient
   $r_{ij}=[z^j] r_i(z)$ in front of $z^j$.\hfill$\Box$

\vspace{2mm}

{\bf Proof 2.} (via duality)

The duality kernel $H$ can be interpreted as a non-singular matrix
$H=(h_{ij})_{i,j\in\nz}$ with entries $h_{ij}=1$ for $j\ge i$ and
$h_{ij}=0$ for $j<i$. The entries of its inverse $H^{-1}=:
(g_{ij})_{i,j\in\nz}$ are given by $g_{ij}=\delta_{i,j}-\delta_{i+1,j}$.
It is known \cite{moehlepitters} that the generator matrix $Q$ of the
block counting process has spectral decomposition
$Q=\tilde{R}\tilde{D}\tilde{L}$, where the matrices
$\tilde{R}=(\tilde{r}_{ij})_{i,j\in\nz}$,
$\tilde{D}=(\tilde{d}_{ij})_{i,j\in\nz}$ and
$\tilde{L}=(\tilde{l}_{ij})_{i,j\in\nz}$ are given by
$\tilde{r}_{ij}=((j-1)!/(i-1)!)|s(i,j)|$,
$\tilde{d}_{ij}=(i-1)\delta_{i,j}$ and
$\tilde{l}_{ij}=(-1)^{i+j}((j-1)!/(i-1)!)S(i,j)$ respectively.
The entries of $D=(d_{ij})_{i,j\in\nz}$ can be read from the diagonal
of $\Gamma$ and are therefore given by $d_{ij}=i\delta_{i,j}$.
Define the matrices $A=(a_{ij})_{i,j\in\nz}$ and $B=(b_{ij})_{i,j\in\nz}$
by $a_{ij}=\delta_{i+1,j}$ and $b_{ij}=\delta_{i-1,j}$.
Clearly $\tilde{D}=BDA$. This together with the duality relation
$H\Gamma^{\top}=QH$ and the spectral decomposition of the block
counting process $Q=\tilde{R}\tilde{D}\tilde{L}$ yields
\[
\Gamma^{\top}
\ =\ H^{-1}\tilde{R}\tilde{D}\tilde{L}H
\ =\ (-H^{-1}\tilde{R}B)D(-A\tilde{L}H).
\]
Hence $\Gamma=RDL$ with $R:=(-A\tilde{L}H)^{\top}$ and
$L:=(-H^{-1}\tilde{R}B)^{\top}$. It remains to calculate the entries of
$R$ and $L$. Using the recursion $S(i+1,j)=jS(i,j)+S(i,j-1)$ we obtain
\begin{eqnarray*}
   r_{ji}
   & =& (-A\tilde{L}H)_{ij}
   \ =\ -(\tilde{L}H)_{i+1,j}
   \ =\ -\sum_{k=1}^j \tilde{l}_{i+1,k}
   \ =\ \sum_{k=1}^j (-1)^{i+k}\frac{(k-1)!}{i!}S(i+1,k)\\
   & = & \sum_{k=1}^j (-1)^{i+k}\frac{k!}{i!}S(i,k)+\sum_{k=1}^j(-1)^{i+k}\frac{(k-1)!}{i!}S(i,k-1)\\
   & = & \sum_{k=1}^j (-1)^{i+k}\frac{k!}{i!}S(i,k)-\sum_{k=0}^{j-1}(-1)^{i+k}\frac{k!}{i!}S(i,k)
   \ = \ (-1)^{i+j}\frac{j!}{i!}S(i,j).
\end{eqnarray*}
Using the recursion $|s(i+1,j+1)|=|s(i,j)|+i|s(i,j+1)|$ we get
\begin{eqnarray*}
   l_{ji}
   & = & (-H^{-1}\tilde{R}B)_{ij}
   \ = \ -(H^{-1}\tilde{R})_{i,j+1}
   \ = \ \tilde{r}_{i+1,j+1}-\tilde{r}_{i,j+1}\\
   & = & \frac{j!}{i!}|s(i+1,j+1)|-\frac{j!}{(i-1)!}|s(i,j+1)|
   \ =\ \frac{j!}{i!}|s(i,j)|.
\end{eqnarray*}
\hfill$\Box$
\end{proof}
\begin{proof} (of Corollary \ref{trans})
   By Theorem \ref{fixationtheo}, $\Gamma=RDL$, where $R$ and $L=R^{-1}$
   have entries (\ref{randl}). Hence, the transition matrix
   $P(t)=e^{t\Gamma}$ has spectral decomposition
   $P(t)=e^{tRDL}=Re^{tD}L$. Thus,
   $
   p_{ij}(t)=\pr(L_t=j\,|\,L_0=i)=(Re^{tD}L)_{ij}
   =\sum_{k=i}^j r_{ik}e^{-\gamma_kt}l_{kj}$. The first formula in (\ref{trans2})
   for $p_{ij}(t)$ follows from $\gamma_k=k$ and from (\ref{randl}).
   Recall that $\alpha:=e^{-t}$.
   Conditional on $L_0=i$ the random variable $L_t$  has probability generating
   function
   \begin{eqnarray*}
      \me(z^{L_t}\,|\,L_0=i)
      & = & \sum_{j=i}^\infty z^j p_{ij}(t)
      \ = \ \sum_{j=i}^\infty z^j (-1)^{i+j}\frac{i!}{j!}
            \sum_{k=i}^j S(k,i) \alpha^k  s(j,k)\\
      & = & (-1)^i i!\sum_{k=i}^\infty S(k,i)\alpha^k\sum_{j=k}^\infty \frac{(-z)^j}{j!}s(j,k)\\
      & = & (-1)^i i!\sum_{k=i}^\infty S(k,i)\alpha^k \frac{(\log(1-z))^k}{k!}\\
      & = & (-1)^i (e^{\alpha\log(1-z)}-1)^i
      \ = \ (1-(1-z)^\alpha)^i,
      \quad |z|<1, t\ge 0, i\in\nz.
   \end{eqnarray*}
   Expansion leads to
   \begin{eqnarray*}
      \me(z^{L_t}\,|\,L_0=i)
      & = & \sum_{k=0}^i {i\choose k}(-1)^k (1-z)^{\alpha k}
      \ =\ \sum_{k=0}^i {i\choose k}(-1)^k \sum_{j=0}^\infty {{\alpha k}\choose j}(-z)^j\\
      & = & \sum_{j=0}^\infty (-z)^j\sum_{k=0}^i (-1)^k {i\choose k}{{\alpha k}\choose j}.
   \end{eqnarray*}
   The coefficient in front of $z^j$ in this expansion yields
   the second formula for $p_{ij}(t)$.\hfill$\Box$
\end{proof}
\begin{proof} (of Corollary \ref{hit})
   The hitting probability $h(i,j)$ is related to the entry
   $g(i,j):=\int_0^\infty \pr(L_t^{(i)}=j)\,{\rm d}t$ of the Green matrix via
   $h(i,j)=\gamma_j g(i,j)=jg(i,j)$ (see, for example, Norris
   \cite[p.~146]{norris}). Thus, for all $i\in\nz$ and $|z|<1$,
   $$
   h_i(z)
   \ :=\ \sum_{j=i}^\infty h(i,j)z^{j-1}
   \ =\ \int_0^\infty \sum_{j=i}^\infty j\pr(L_t^{(i)}=j)z^{j-1}\,{\rm d}t
   \ =\ \int_0^\infty \frac{{\rm d}}{{\rm d}z}\sum_{j=i}^\infty \pr(L_t^{(i)}=j)z^j\,{\rm d}t.
   $$
   Plugging in the formula (\ref{trans1}) for the pgf of $L_t^{(i)}$ it follows that
   $$
   h_i(z)
   \ =\ \int_0^\infty \frac{{\rm d}}{{\rm d}z}(1-(1-z)^{e^{-t}})^i\,{\rm d}t
   \ =\ \int_0^\infty i(1-(1-z)^{e^{-t}})^{i-1} e^{-t}(1-z)^{e^{-t}-1}\,{\rm d}t.
   $$
   Substituting $x:=e^{-t}$ and noting that ${\rm d}t/{\rm d}x=-1/x$ leads to
   $h_i(z)
   =(1-z)^{-1}\int_0^1 i(1-(1-z)^x)^{i-1}(1-z)^x\,{\rm d}x$.
   Substituting further $y:=1-(1-z)^x$ and noting that ${\rm d}x/{\rm d}y=
   1/((1-y)(-\log(1-z)))$ we obtain
   $$
   h_i(z)\ =\ \frac{1}{(1-z)(-\log(1-z))}\int_0^z iy^{i-1}\,{\rm d}y
   \ =\ \frac{z^i}{(1-z)(-\log(1-z))},\qquad i\in\nz, |z|<1.
   $$
   In particular, $h(i,j)=h(1,j-i+1)$.
   The asymptotic expansion (\ref{hitasy}) follows from
   Panholzer \cite[Eq.~(19)]{panholzer}. Formula (\ref{hit0}) is obtained
   as follows. Let $(J_k)_{k\in\nz_0}$ denote the jump chain of the fixation
   line $(L_t)_{t\ge 0}$. Given this chain is in state $i$ it jumps to state $i+j$ with
   probability $\gamma_{i,i+j}/\gamma_i=1/(j(j+1))=:u_j$, $j\in\nz$. From this
   property it is easily seen that the jump chain has independent increments,
   i.e. $J_0=1$, $J_1=1+\eta_1$, $J_2=1+\eta_1+\eta_2$ and so on,
   where $\eta_1,\eta_2,\ldots$ are iid random variables with distribution
   $\pr(\eta_1=j)=u_j$, $j\in\nz$. For $1\le i<j$ it follows that
   $h(i,j)
   = h(1,j-i+1)
    = \sum_{k=1}^{j-i}\pr(J_k=j-i+1)
    = \sum_{k=1}^{j-i}\pr(\eta_1+\cdots+\eta_k=j-i)$.
   Formula (\ref{hit1a}) for $h(i,j)$
   follows from $h(i,j)=jg(i,j)=j\int_0^\infty\pr(L_t^{(i)}=j)\,{\rm d}t$ and
   \begin{eqnarray*}
      \int_0^\infty \pr(L_t^{(i)}=j)\,{\rm d}t
      & = & \int_0^\infty (-1)^{i+j}\frac{i!}{j!}\sum_{k=i}^j S(k,i)e^{-tk}
            s(j,k)\,{\rm d}t\\
      & = & (-1)^{i+j}\frac{i!}{j!}\sum_{k=i}^j\frac{S(k,i)s(j,k)}{k}.
   \end{eqnarray*}
   Eq.~(\ref{hit1b}) follows from $h(i,j)=h(1,j-i+1)$ and $S(k,1)=1$
   for all $k\in\nz$.
   Moreover, for $i=1$ we have $\pr(L_t=j)=\alpha\Gamma(j-\alpha)/(j!\Gamma(1-\alpha))$ with
   $\alpha:=e^{-t}$. Thus,
   $$
   g(1,j)\ =\ \int_0^\infty \pr(L_t=j)\,{\rm d}t
   \ =\ \frac{1}{j!}\int_0^1 \frac{\Gamma(j-\alpha)}{\Gamma(1-\alpha)}\,{\rm d}\alpha
   \ =\ \frac{1}{j!}\int_0^1 \frac{\Gamma(j-1+x)}{\Gamma(x)}\,{\rm d}x
   $$
   and, hence, we obtain the integral representation
   $$
   h(i,j)
   \ =\ h(1,j-i+1)
   \ =\ (j-i+1)g(1,j-i+1)
   \ =\ \frac{1}{(j-i)!}\int_0^1 \frac{\Gamma(j-i+x)}{\Gamma(x)}\,{\rm d}x,
   \quad 1\le i\le j.
   $$
   The last formula for $h(i,j)$ in (\ref{hit2}) follows from
   $\Gamma(n+x)/\Gamma(x)=\sum_{k=0}^n |s(n,k)|x^k$, $n\in\nz_0$, $x\in\rz$.
   The proof of Corollary \ref{hit} is complete.\hfill$\Box$
\end{proof}
\begin{remark}
%
% Internal note for the authors:
% The following values h(i,j) are computed via the program hit.mws
%
   Note that $\pr(\eta_1+\cdots+\eta_k=j-i)=
   \sum_{i_1,\ldots,i_k} u_{i_1}\cdots u_{i_k}$,
   where the sum extends over all $i_1,\ldots,i_k\in\nz$ satisfying
   $i_1+\cdots+i_k=j-i$.
   Hence, concrete values of the hitting probabilities are
   $h(1,1)=1$, $h(1,2)=\pr(\eta_1=1)=u_1=1/2$,
   $h(1,3)=\pr(\eta_1=2)+\pr(\eta_1+\eta_2=2)=u_2+u_1^2=1/6+1/4=
   5/12\approx 0.41667$,
   $h(1,4)=\pr(\eta_1=3)+\pr(\eta_1+\eta_2=3)+\pr(\eta_1+\eta_2+\eta_3=3)
   =u_3+2u_1u_2+u_1^3=1/12+1/6+1/8=3/8=0.375$,
   $h(1,5)
   =u_4+(2u_1u_3+u_2^2)+3u_1^2u_2
   =1/20+1/9+1/8
   =251/720\approx 0.34861$,
   $h(1,6)=95/288\approx 0.32986$,
   $h(1,7)=19087/60480\approx 0.31559$
   and so on.
\end{remark}
\begin{proof} (of Corollary \ref{dist})
   By the definition of $\tau_{ni}$ and the duality of
   $(N_t)_{t\ge 0}$ and $(L_t)_{t\ge 0}$ we have
   $\pr(\tau_{ni}\le t)=\pr(N_t^{(n)}\le i)=\pr(L_t^{(i)}\ge n)
   =\sum_{j=n}^\infty p_{ij}(t)$. Using the second formula for
   $p_{ij}(t)$ in (\ref{trans2}) yields
   \begin{eqnarray*}
      \pr(\tau_{ni}\le t)
      & = & \sum_{j=n}^\infty (-1)^j\sum_{k=1}^i (-1)^k
            {i\choose k}{{e^{-t}k}\choose j}\\
      & = & \sum_{k=1}^i (-1)^k {i\choose k}\sum_{j=n}^\infty (-1)^j{{e^{-t}k}\choose j}\\
      & = & \sum_{k=1}^i (-1)^k {i\choose k} (-1)^n {{e^{-t}k-1}\choose{n-1}},
   \end{eqnarray*}
   where the last equality holds since
   $\sum_{j=n}^\infty (-1)^j{z\choose j}=(-1)^n{{z-1}\choose{n-1}}$ for all
   $n\in\nz$ and all $z\in\rz$.

   Fix $x\in\rz$ and define $F(x):=e^{-e^{-x}}$ for convenience. Assume that
   $n$ is sufficiently large such that
   $x+\log\log n>0$. Choosing $t:=x+\log\log n$
   and noting that for all sufficiently large $n$
   \begin{eqnarray*}
      (-1)^{n-1} {{e^{-t}k-1}\choose{n-1}}
      & = & \frac{\Gamma(n-ke^{-x}/\log n)}{\Gamma(n)\Gamma(1-ke^{-x}/\log n)}\\
      & \sim & \frac{\Gamma(n-ke^{-x}/\log n)}{\Gamma(n)}
      \ \to\ e^{-ke^{-x}}\ =\ (F(x))^k
   \end{eqnarray*}
   as $n\to\infty$ by an application of Stirling's formula
   $\Gamma(n+1)\sim (n/e)^n\sqrt{2\pi n}$ as $n\to\infty$, it follows that
   \begin{eqnarray*}
      \pr(\tau_{ni}-\log\log n\le x)
      & = & \pr(\tau_{ni}\le x+\log\log n)\\
      & \to & \sum_{k=1}^i (-1)^{k-1}{i\choose k} (F(x))^k
      \ =\ 1-(1-F(x))^i,\qquad n\to\infty.
   \end{eqnarray*}
   It remains to note that $x\mapsto 1-(1-F(x))^i$, $x\in\rz$,
   is the distribution function of the minimum of $i$
   standard Gumbel distributed random variables.\hfill$\Box$
\end{proof}
Before we will prove Corollary \ref{edgeworth} we provide the Taylor
expansion of the map $x\mapsto 1/\Gamma(1-x)$.
\begin{lemma} \label{coef}
   The map $x\mapsto 1/\Gamma(1-x)$ has Taylor expansion
   $1/\Gamma(1-x)=\sum_{k=0}^\infty c_kx^k$, $|x|<1$, where the coefficients
   $c_0,c_1,\ldots$ are related to the moments $m_k=(-1)^k\Gamma^{(k)}(1)$,
   $k\in\nz_0$, of the Gumbel distribution via $c_0=m_0=1$ and
   \begin{equation} \label{ck}
      c_k\ =\
      \sum_{j=1}^k (-1)^j\sum_{{k_1,\ldots,k_j\in\nz}\atop{k_1+\cdots+k_j=k}}
      \frac{m_{k_1}\cdots m_{k_j}}{k_1!\cdots k_j!},\qquad k\in\nz.
   \end{equation}
   Alternatively,
   \begin{equation} \label{ckalt}
      c_k\ =\ \frac{(-1)^k}{k!}\sum_{l=1}^k (-1)^l {{k+1}\choose{l+1}}(\Gamma^l)^{(k)}(1)
      \qquad k\in\nz,
   \end{equation}
   where $(\Gamma^l)^{(k)}$ denotes the $k$th derivative of
   the $l$th power of $\Gamma$.
\end{lemma}
\begin{remark}
   Concrete values are
   $c_1=-m_1=-\gamma\approx-0.577216$,
   $c_2=m_1^2-m_2/2=\gamma^2-(\gamma^2+\zeta(2))/2=\gamma^2/2-\pi^2/12\approx -0.655878$,
   $c_3=-m_3/6+m_1m_2-m_1^3=\gamma\zeta(2)/2-\zeta(3)/3-\gamma^3/6=
   \pi^2\gamma/12-\zeta(3)/3-\gamma^3/6\approx 0.042003$ and so on.
\end{remark}
\begin{proof}
   A Gumbel distributed random variable $\tau$ has moment generating
   function $\me(e^{x\tau})=\Gamma(1-x)$, $x<1$. Thus, the map
   $x\mapsto\Gamma(1-x)$ has Taylor expansion $\Gamma(1-x)=\sum_{k=0}^\infty
   a_k x^k$, $|x|<1$,where $a_k:=m_k/k!$ and $m_k=\me(\tau^k)$, $k\in\nz_0$,
   are the moments of the Gumbel distribution.
   For the reciprocal map $1/\Gamma(1-x)$ it follows that
   \begin{eqnarray*}
      \frac{1}{\Gamma(1-x)}
      & = & \sum_{j=0}^\infty (1-\Gamma(1-x))^j
      \ = \ \sum_{j=0}^\infty \bigg(\sum_{k=1}^\infty -a_kx^k\bigg)^j\\
      & = & 1 + \sum_{j=1}^\infty \sum_{k_1,\ldots,k_j\in\nz}
            (-a_{k_1})\cdots(-a_{k_j}) x^{k_1+\cdots+k_j}\\
      & = & 1 + \sum_{j=1}^\infty (-1)^j\sum_{k=1}^\infty x^k
            \sum_{{k_1,\ldots,k_j\in\nz}\atop{k_1+\cdots+k_j=k}}
            a_{k_1}\cdots a_{k_j}
      \ = \ \sum_{k=0}^\infty c_k x^k
   \end{eqnarray*}
   with $c_0:=1$ and $c_k$, $k\in\nz$, as given in (\ref{ck}),
   since $a_k=m_k/k!$, $k\in\nz_0$. Since $m_k=(-1)^k\Gamma^{(k)}(1)$,
   (\ref{ck}) can be rewritten as
   \begin{eqnarray*}
      c_k
      & = & \sum_{j=1}^k (-1)^{j+k}
            \sum_{{k_1,\ldots,k_j\in\nz}\atop{k_1+\cdots+k_j=k}}
            \frac{\Gamma^{(k_1)}(1)\cdots \Gamma^{(k_j)}(1)}{k_1!\cdots k_j!}\\
      & = & \sum_{j=1}^k \frac{(-1)^{j+k}}{k!}
            \sum_{l=1}^j (-1)^{j-l}{j\choose l} (\Gamma^l)^{(k)}(1),
            \qquad k\in\nz,
   \end{eqnarray*}
   where the last equality holds by Lemma 1 in the appendix of
   \cite{moehlelooking}. Interchanging the sums and noting
   that $\sum_{j=l}^k {j\choose l}={{k+1}\choose{l+1}}$ yields (\ref{ckalt}).\hfill$\Box$
\end{proof}
\begin{proof} (of Corollary \ref{edgeworth})
   Fix $x\in\rz$ and define $F(x):=e^{-e^{-x}}$.
   By Corollary \ref{dist}, for all sufficiently large $n$,
   \begin{equation} \label{edge1}
      \pr(\tau_{ni}-\log\log n\le x)
      \ =\ \sum_{j=1}^i (-1)^{j-1}{i\choose j}
      \frac{\Gamma(n-je^{-x}/\log n)}{\Gamma(n)\Gamma(1-je^{-x}/\log n)}.
   \end{equation}
   For every $c\in\rz$ it is easily checked that $\Gamma(n+c/\log n)/\Gamma(n)
   =e^c+O(1/(n\log n))$ as $n\to\infty$. For $c=-je^{-x}$ we obtain
   \begin{equation} \label{edge2}
      \frac{\Gamma(n-je^{-x}/\log n)}{\Gamma(n)}
      \ =\ (F(x))^j + O\bigg(\frac{1}{n\log n}\bigg).
   \end{equation}
   Moreover (see Lemma \ref{coef}), from $1/\Gamma(1-x)=
   \sum_{k=0}^\infty c_kx^k$ we conclude that, for all $K\in\nz_0$,
   \begin{equation} \label{edge3}
      \frac{1}{\Gamma(1-je^{-x}/\log n)}
      \ =\ \sum_{k=0}^K c_k\bigg(\frac{je^{-x}}{\log n}\bigg)^k
      + O\bigg(\frac{1}{(\log n)^{K+1}}\bigg).
   \end{equation}
   Multiplying (\ref{edge2}) with (\ref{edge3}) yields
   $$
   \frac{\Gamma(n-je^{-x}/\log n)}{\Gamma(n)\Gamma(1-je^{-x}/\log n)}
   \ =\ (F(x))^j\sum_{k=0}^K c_k\bigg(\frac{je^{-x}}{\log n}\bigg)^k
   + O\bigg(\frac{1}{(\log n)^{K+1}}\bigg).
   $$
   Plugging this expansion into (\ref{edge1}) and exchanging the sums yields
   $$
   \pr(\tau_{ni}-\log\log n\le x)
   \ =\ \sum_{k=0}^K c_k\bigg(\frac{e^{-x}}{\log n}\bigg)^k
        \sum_{j=1}^i (F(x))^j(-1)^{j-1}{i\choose j} j^k
        + O\bigg(\frac{1}{(\log n)^{K+1}}\bigg),
   $$
   which is the desired Edgeworth expansion
   with coefficients $d_{ki}(x)$ as defined in (\ref{dki1}). It remains
   to verify the alternative representation (\ref{dki2}) of the coefficients
   $d_{ki}(x)$. It is readily checked by induction on $k\in\nz_0$
   that $(t\frac{\partial}{\partial t})^k f(t)
   =\sum_{j=0}^k S(k,j)t^jf^{(j)}(t)$ for every $k$-times differentiable
   function $f:\rz\to\rz$, where the $S(k,j)$ denote the Stirling numbers of
   the second kind. Applying this formula to $f(t):=1-(1-t)^i$ with $i\in\nz$
   it follows for all $k\in\nz_0$ and $t\in\rz$ that
   \begin{eqnarray*}
      \sum_{j=1}^i (-1)^{j-1}{i\choose j}j^kt^j
      & = & \bigg(t\frac{\partial}{\partial t}\bigg)^k\sum_{j=1}^i (-1)^{j-1}{i\choose j}t^j
      \ = \ \bigg(t\frac{\partial}{\partial t}\bigg)^k (1-(1-t)^i)\\
      & = & \sum_{j=0}^k S(k,j) t^j\bigg(\frac{\partial}{\partial t}\bigg)^j(1-(1-t)^i)\\
      & = & S(k,0)(1-(1-t)^i) + \sum_{j=1}^k S(k,j)t^j(-1)^{j-1}
      (i)_j(1-t)^{i-j},
   \end{eqnarray*}
   where $(i)_j:=i(i-1)\cdots(i-j+1)$.
   Replacing $t$ by $F(x)$ and noting that $S(k,0)=0$ for $k\in\nz$
   shows that (\ref{dki1}) coincides for $k\in\nz$ with (\ref{dki2}).
   \hfill$\Box$
\end{proof}
\begin{proof} (of Theorem \ref{main} a))
   Let $Z^{(n)}:=(X_t^{(n)},t)_{t\ge 0}$ and $Z:=(X_t,t)_{t\ge 0}$ denote
   the space-time processes of $X^{(n)}=(X_t^{(n)})_{t\ge 0}$ and
   $X=(X_t)_{t\ge 0}$ respectively. Note that $Z^{(n)}$ has state space
   $S_n:=\{(j/n^{e^{-t}},t)\,:\,j\in\{1,\ldots,n\},t\ge 0\}=
   \bigcup_{t\ge 0} (E_{n,t}\times\{t\})$, where $E_{n,t}:=
   \{j/n^{e^{-t}}\,:\,j\in\{1,\ldots,n\}\}$, and that $Z$
   has state space $S:=E\times [0,\infty)=[0,\infty)^2$. The processes
   $Z^{(n)}$ and $Z$ are time-homogeneous (see, for example,
   Revuz and Yor \cite[p.~85, Exercise (1.10)]{revuzyor}).
   In the following it is shown that $Z^{(n)}$ converges in $D_S[0,\infty)$
   to $Z$ as $n\to\infty$. Note that this convergence implies the desired
   convergence of $X^{(n)}$ in $D_E[0,\infty)$ to $X$ as $n\to\infty$.
   Define $\pi_n:B(S)\to B(S_n)$ via
   $\pi_nf(x,s):=f(x,s)$ for all $f\in B(S)$ and $(x,s)\in S_n$.
   By Proposition \ref{laplaceprop} it suffices to verify that, for every
   $t\ge 0$ and $\lambda,\mu>0$,
   $$
   \lim_{n\to\infty} \sup_{s\ge 0}\sup_{x\in E_{n,s}}
   |T_t^{(n)}\pi_n f_{\lambda,\mu}(x,s)-\pi_nT_t f_{\lambda,\mu}(x,s)|\ =\ 0,
   $$
   where $(T_t^{(n)})_{t\ge 0}$ and $(T_t)_{t\ge 0}$ denote the semigroups
   of the space-time processes $Z^{(n)}$ and $Z$ respectively
   and the test functions $f_{\lambda,\mu}:S\to\rz$ are defined via
   $f_{\lambda,\mu}(x,s):=e^{-\lambda x-\mu s}$ for all $(x,s)\in S$.
   Fix $t\ge 0$ and $\lambda,\mu>0$.
   For convenience, define $\alpha:=e^{-t}$ and $\beta:=e^{-s}$.
   We have
   \begin{eqnarray*}
      T_t^{(n)}\pi_n f_{\lambda,\mu}(x,s)
      & = & \me(f_{\lambda,\mu}(X_{s+t}^{(n)},s+t)\,|\,X_s^{(n)}=x)\\
      & = & (\alpha\beta)^\mu\me(\exp(-\lambda/n^{\alpha\beta} N_{s+t}^{(n)})\,|\,N_s^{(n)}=xn^\beta)\\
      & = & (\alpha\beta)^\mu\me(\exp(-\lambda/n^{\alpha\beta} N_t^{(xn^\beta)})),
            \qquad (x,s)\in S_n,
   \end{eqnarray*}
   and
   \begin{eqnarray*}
      \pi_nT_tf_{\lambda,\mu}(x,s)
      & = & \me(f_{\lambda,\mu}(X_{s+t},s+t)\,|\,X_s=x)\\
      & = & (\alpha\beta)^\mu\me(\exp(-\lambda X_{s+t})\,|\,X_s=x)\\
      & = & (\alpha\beta)^\mu\me(\exp(-\lambda x^\alpha X_t)),
      \qquad (x,s)\in S.
   \end{eqnarray*}
   Thus, we have to verify that
   $$
   \lim_{n\to\infty}\sup_{s\ge 0}\sup_{x\in E_{n,s}}
   (\alpha\beta)^\mu|\me(\exp(-\lambda/n^{\alpha\beta}N_t^{(xn^\beta)}  ))
   -\me(\exp(-\lambda x^\alpha X_t))|
   \ =\ 0.
   $$
   Since both expectations are bounded between $0$ and $1$ and since
   $(\alpha\beta)^\mu=e^{-\mu(s+t)}$ tends
   to $0$ as $s\to\infty$ it suffices to verify that, for every $s_0>0$,
   $$
   \lim_{n\to\infty}\sup_{s\in [0,s_0]}\sup_{x\in E_{n,s}}
   |\me(\exp(-\lambda/n^{\alpha\beta}N_t^{(xn^\beta)}  ))-\me(\exp(-\lambda x^\alpha X_t))|
   \ =\ 0.
   $$
   We will even verify that
   $$
   \lim_{n\to\infty}
   \sup_{s\in [0,s_0]}\sup_{x\ge 0}
   |\me(\exp(-\lambda/n^{\alpha\beta} N_t^{(\lfloor xn^\beta\rfloor)}    ))
   -\me(\exp(-\lambda x^\alpha X_t))|
   \ =\ 0.
   $$
   The difference of the two expectations depends on $n$ and $s$ only via
   $n^\beta=n^{e^{-s}}$. Since the map
   $s\mapsto n^{e^{-s}}$ is non-increasing it follows that the
   convergence for fixed $s\in [0,s_0]$ is slower
   as $s$ is larger. So the slowest convergence holds at the right
   end point $s=s_0$. Thus, it suffices to verify that, for every
   $s\ge 0$,
   $$
   \lim_{n\to\infty} \sup_{x\ge 0}
   |\me(\exp(-\lambda/n^{\alpha\beta}N_t^{(\lfloor xn^\beta\rfloor)}))
   -\me(\exp(-\lambda x^\alpha X_t))|
   \ =\ 0.
   $$
   The map $x\mapsto\me(\exp(-\lambda x^\alpha X_t))$ is bounded, continuous,
   and non-increasing. Moreover, for every $n$ the map $x\mapsto
   \me(\exp(-\lambda/n^{\alpha\beta}N_t^{(\lfloor xn^\beta\rfloor)}))$ is
   non-increasing. Thus, by the theorem of P\'olya, it suffices to verify that, for
   every $s\ge 0$ and $x\ge 0$,
   $$
   \lim_{n\to\infty}
   \me(\exp(-\lambda/n^{\alpha\beta} N_t^{(\lfloor xn^\beta\rfloor)}))
   \ =\ \me(\exp(-\lambda x^\alpha X_t)).
   $$
   Note that we have reduced the problem to verify the convergence
   uniformly for all $s\ge 0$ and $x\in E_{n,s}$ to the problem to
   verify the convergence pointwise for all points $(s,x)\in [0,\infty)^2$.

  Define $\tau:=n^\beta$. Using this notation it remains to verify that
   \begin{equation} \label{tobeshown}
   \lim_{\tau\to\infty}
   \me(\exp(-\lambda/\tau^\alpha N_t^{(\lfloor x\tau\rfloor)}))
   \ =\ \me(\exp(-\lambda x^\alpha X_t)).
   \end{equation}
   We have
   \begin{eqnarray*}
      \me(\exp(-\lambda/\tau^\alpha N_t^{(\lfloor x\tau\rfloor)}))
      & = & \sum_{m=0}^\infty \frac{(-\lambda)^m}{m!}
            \frac{\me((N_t^{(\lfloor x\tau\rfloor)})^m)}
            {\tau^{\alpha m}}.
   \end{eqnarray*}
   Note that the series on the right hand side is absolutely convergent,
   since $N_t^{(\lfloor x\tau\rfloor)}\le x\tau$ and, hence,
   $\me((N_t^{(\lfloor x\tau\rfloor)})^m)\le (x\tau)^m$.
   Applying the formula $z^m=\sum_{i=0}^m (-1)^{m-i}S(m,i)[z]_i$,
   $m\in\nz_0$, $z>0$, where $[z]_i:=\Gamma(z+i)/\Gamma(z)$ for $z,i>0$,
   it follows that
   \begin{eqnarray*}
      \frac{\me((N_t^{(\lfloor x\tau\rfloor)})^m)}{\tau^{\alpha m}}
      & = & \sum_{i=0}^m (-1)^{m-i}S(m,i)\frac{\me([N_t^{(\lfloor x\tau\rfloor)}]_i)}{\tau^{\alpha m}}
      \ = \ \sum_{i=0}^m (-1)^{m-i}S(m,i)\me(X_t^i)\frac{[\lfloor x\tau\rfloor]_{\alpha i}}{\tau^{\alpha m}}
   \end{eqnarray*}
   by Lemma 3.1 of \cite{moehlemittag}. From $[\lfloor x\tau\rfloor]_{\alpha i}
   \sim (x\tau)^{\alpha i}=x^{\alpha i}\tau^{\alpha i}$
   as $\tau\to\infty$ we conclude that
   only the summand $i=m$ yields asymptotically a non-zero contribution
   and it follows that
   $$
   \lim_{\tau\to\infty}
   \frac{\me((N_t^{(\lfloor x\tau\rfloor)})^m)} {\tau^{\alpha m}}
   \ =\ \me(X_t^m) x^{\alpha m}
   \ =\ \me((x^\alpha X_t)^m).
   $$
   Moreover,
   $$
   \frac{\me((N_t^{(\lfloor x\tau\rfloor)})^m)}{\tau^{\alpha m}}
   \ \le\ \frac{\me([N_t^{(\lfloor x\tau\rfloor)}]_m)}{\tau^{\alpha m}}
   \ =\ \me(X_t^m)
        \frac{[\lfloor x\tau\rfloor]_{\alpha m}}{\tau^{\alpha m}}
   \ \le\ \me(X_t^m)\frac{[x\tau]_{\alpha m}}{\tau^{\alpha m}}.
   $$
   It is readily checked that the map $\tau\mapsto [x\tau]_{\alpha m}/\tau^{\alpha m}$
   is non-increasing in $\tau$. Thus, we obtain the upper bound
   $$
   \frac{\me((N_t^{(\lfloor x\tau\rfloor)})^m)}{\tau^{\alpha m}}
   \ \le\ \me(X_t^m)\frac{[x\tau_0]_{\alpha m}}{\tau_0^{\alpha m}}
   \quad\mbox{for all $\tau\ge\tau_0$.}
   $$
   Note that
   $$
   \frac{\lambda^m}{m!}\me(X_t^m)\frac{[x\tau_0]_{\alpha m}}{\tau_0^{\alpha m}}
   \ =\ \frac{\lambda^m}{m!}\frac{m!}{\Gamma(1+\alpha m)}
   \frac{\Gamma(x\tau_0+\alpha m)}{\tau_0^{\alpha m}\Gamma(x\tau_0)}
   \ \sim\ \bigg(\frac{\lambda}{\tau_0^\alpha}\bigg)^m (\alpha m)^{x\tau_0-1}
   $$
   as $m\to\infty$. Thus, if we choose $\tau_0$ sufficiently large such
   that $\lambda/\tau_0^\alpha<1$, for example $\tau_0:=(2\lambda)^{1/\alpha}$,
   then the dominating map
   $m\mapsto(\lambda^m/m!)\me(X_t^m)[x\tau_0]_{\alpha m}/\tau_0^{\alpha m}$
   is integrable with respect to the counting measure on $\nz$.
   Thus, it is allowed to apply the dominated convergence theorem, which yields
   $$
   \lim_{\tau\to\infty}\me(\exp(-\lambda/\tau^\alpha N_t^{(\lfloor x\tau\rfloor)}))
   \ =\ \sum_{m=0}^\infty \frac{(-\lambda)^m}{m!} \me((x^\alpha X_t)^m)
   \ =\ \me(\exp(-\lambda x^\alpha X_t)).
   $$
   Thus, (\ref{tobeshown}) is established. The proof is complete.\hfill$\Box$
\end{proof}
Before we come to the proof of Theorem \ref{main} b), we provide a recursion
for the Laplace transforms of the finite-dimensional distributions of Neveu's
continuous-state branching process $Y=(Y_t)_{t\ge 0}$.
\begin{lemma}[Recursion for the Laplace transforms of ${\mathbf Y}$] \label{lap}
   Let $0=t_0\le t_1<t_2<\cdots.$ For $k\in\nz$ let
   $\psi_k:[0,\infty)^k\to [0,1]$, defined via
   $\psi_k(\lambda_1,\ldots,\lambda_k)
   :=\me(e^{-\lambda_1Y_{t_1}}\cdots e^{-\lambda_kY_{t_k}})$ for all
   $\lambda_1,\ldots,\lambda_k\ge 0$,
   denote the Laplace transform of $Y_{t_1},\ldots,Y_{t_k}$.
   Then, $\psi_k$ satisfies the recursion $\psi_1(\lambda_1)=e^{-\lambda_1^{\alpha_1}}$
   for all $\lambda_1\ge 0$ and
   $$
   \psi_k(\lambda_1,\ldots,\lambda_k)
   \ =\ \psi_{k-1}(\lambda_1,\ldots,\lambda_{k-2},\lambda_{k-1}+\lambda_k^{\alpha_k/\alpha_{k-1}}),
   \qquad\ k\in\nz\setminus\{1\}, \lambda_1,\ldots,\lambda_k\ge 0,
   $$
   where $\alpha_j:=e^{-t_j}$, $1\le j\le k$.
\end{lemma}
\begin{proof} (of Lemma \ref{lap})
   Clearly, $\psi_1(\lambda_1)=\me( e^{-\lambda_1 Y_{t_1}})
   =e^{-\lambda_1^{\alpha_1}}$ for all $\lambda_1\ge 0$.
   Moreover, for all $\lambda_1,\ldots,\lambda_k\ge 0$,
   \begin{eqnarray*}
      \psi_k(\lambda_1,\ldots,\lambda_k)
      & = & \me(e^{-\lambda_1Y_{t_1}}\cdots e^{-\lambda_kY_{t_k}})\\
      & = & \me(\me(
               e^{-\lambda_1Y_{t_1}}\cdots e^{-\lambda_kY_{t_k}}
            \,|\,Y_{t_1},\ldots,Y_{t_{k-1}}))\\
      & = & \me(
               e^{-\lambda_1Y_{t_1}}\cdots e^{-\lambda_{k-1}Y_{t_{k-1}}}
               \me(e^{-\lambda_kY_{t_k}}|Y_{t_{k-1}})
            ).
   \end{eqnarray*}
   Since $\me(e^{-\lambda_k Y_{t_k}}|Y_{t_{k-1}})=e^{-\lambda_k^{\alpha_k/\alpha_{k-1}}Y_{t_{k-1}}}$
   almost surely it follows that
   \begin{eqnarray*}
      \psi_k(\lambda_1,\ldots,\lambda_k)
      & = & \me(
         e^{\lambda_1Y_{t_1}}\cdots e^{-\lambda_{k-2}Y_{t_{k-2}}}
         e^{-(\lambda_{k-1}+\lambda_k^{\alpha_k/\alpha_{k-1}})Y_{t_{k-1}}}
      )\\
      & = & \psi_{k-1}(\lambda_1,\ldots,\lambda_{k-2},\lambda_{k-1}+\lambda_k^{\alpha_k/\alpha_{k-1}}).
   \end{eqnarray*}

   \vspace{-6mm}

    \hfill$\Box$
\end{proof}
We are now able to verify Theorem \ref{main} b).
\begin{proof} (of Theorem \ref{main} b))
   The proof is divided into two parts. First the convergence of the
   finite-dimensional distributions is verified.
   Afterwards the convergence in $D_E[0,\infty)$ is considered.
   In fact Part 2 does not use results from Part 1, so one could omit
   Part 1. However, we think it is helpful for the reader to consider
   first the convergence of the finite-dimensional distributions.

   \vspace{2mm}

   {\bf Part 1.} (Convergence of the finite-dimensional distributions)
   Fix $0=t_0\le t_1<t_2<\cdots$. For $k,n\in\nz$ let $\psi_k^{(n)}
   :[0,\infty)^k\to [0,1]$ and $\psi_k:[0,\infty)^k\to [0,1]$ denote
   the Laplace transforms of $(Y_{t_1}^{(n)},\ldots,Y_{t_k}^{(n)})$ and
   $(Y_{t_1},\ldots,Y_{t_k})$ respectively.
   In the following the pointwise convergence $\psi_k^{(n)}\to\psi_k$
   as $n\to\infty$ is verified by induction on $k\in\nz$.

   Obviously, $L_{t_1}^{(n)}$ has generating function
   $\me(z_1^{L_{t_1}^{(n)}})=(1-(1-z_1)^{\alpha_1})^n$, $z_1\in [0,1]$,
   where $\alpha_1:=e^{-t_1}$. Replacing $z_1$ by $e^{-\lambda_1/n^{1/\alpha_1}}$
   with $\lambda_1\ge 0$ it follows that
   $$
   \psi_1^{(n)}(\lambda_1)
   \ =\ \me(e^{-\lambda_1 Y_{t_1}^{(n)}})
   \ =\ (1-(1-e^{-\lambda_1/n^{1/\alpha_1}})^{\alpha_1})^n.
   $$
   Clearly, $\psi_1(\lambda_1)=\me(e^{-\lambda_1 Y_{t_1}})=e^{-\lambda_1^{\alpha_1}}$.
   Using the shortage $x:=\lambda_1/n^{1/\alpha_1}$ and the
   inequality $|a^n-b^n|\le n|a-b|$, $|a|,|b|\le 1$, it follows that
   \begin{eqnarray*}
      |\psi_1^{(n)}(\lambda_1) - \psi_1(\lambda_1)|
      & = & |(1-(1-e^{-x})^{\alpha_1})^n - (e^{-x^{\alpha_1}})^n|\\
      & \le & n|1-(1-e^{-x})^{\alpha_1} - e^{-x^{\alpha_1}}|
      \ = \ n(e^{-x^{\alpha_1}}-1+(1-e^{-x})^{\alpha_1}),
   \end{eqnarray*}
   since $(1-e^{-x})^{\alpha_1}\ge 1-e^{-x^{\alpha_1}}$ by Lemma \ref{inequality}.
   From $1-e^{-x}\le x$, $x\in\rz$, and $e^{-t}-1+t\le t^2/2$, $t\ge 0$,
   we conclude that
   $$
      |\psi_1^{(n)}(\lambda_1)-\psi_1(\lambda_1)|
      \ \le\ n(e^{-x^{\alpha_1}}-1+x^{\alpha_1})
      \ \le\ n\frac{(x^{\alpha_1})^2}{2}
      \ =\ \frac{\lambda_1^{2\alpha_1}}{2n}
      \ \to\ 0,\qquad n\to\infty.
   $$
   Thus, the pointwise convergence $\psi_1^{(n)}\to\psi_1$
   as $n\to\infty$ is established.

   Now fix $k\in\nz\setminus\{1\}$. The induction step from $k-1$ to $k$
   works as follows. For convenience define $\alpha_j:=e^{-t_j}$ for
   all $j\in\nz$. For all $z_1,\ldots,z_k\in [0,1]$,
   \begin{eqnarray*}
      \me(z_1^{L_{t_1}^{(n)}}\cdots z_k^{L_{t_k}^{(n)}})
      & = & \me(\me(z_1^{L_{t_1}^{(n)}}\cdots z_k^{L_{t_k}^{(n)}}
            \,|\,L_{t_1}^{(n)},\ldots,L_{t_{k-1}}^{(n)}))\\
      & = & \me(z_1^{L_{t_1}^{(n)}}\cdots z_{k-1}^{L_{t_{k-1}}^{(n)}}
            \me(z_k^{L_{t_k}^{(n)}}\,|\,L_{t_{k-1}}^{(n)})).
   \end{eqnarray*}
   Since $\me(z_k^{L_{t_k}^{(n)}}\,|\,L_{t_{k-1}}^{(n)})=(1-(1-z_k)^{\alpha_k/\alpha_{k-1}})^{L_{t_{k-1}}^{(n)}}$
   almost surely it follows that
   $$
   \me(z_1^{L_{t_1}^{(n)}}\cdots z_k^{L_{t_k}^{(n)}})
   \ =\ \me(z_1^{L_{t_1}^{(n)}}\cdots z_{k-2}^{L_{t_{k-2}}^{(n)}}
   u_{k-1}^{L_{t_{k-1}}^{(n)}}),
   $$
   where % $u_j:=z_j$ for all $j\in\{1,\ldots,k-2\}$ and
   $u_{k-1}:=z_{k-1}(1-(1-z_k)^{\alpha_k/\alpha_{k-1}})$.
   Replacing for each $j\in\{1,\ldots,k\}$ the variable $z_j$ by
   $e^{-\lambda_j/n^{1/\alpha_j}}$ with $\lambda_j\ge 0$
   it follows that
   \begin{eqnarray}
      \psi_k^{(n)}(\lambda_1,\ldots,\lambda_k)
      & = & \me(e^{-\lambda_1 Y_{t_1}^{(n)}}\cdots e^{-\lambda_kY_{t_k}^{(n)}})\nonumber\\
      & = & \me(e^{-\lambda_1 Y_{t_1}^{(n)}}\cdots e^{-\lambda_{k-2}Y_{t_{k-2}}^{(n)}}e^{-\mu_{k-1}(n)Y_{t_{k-1}}^{(n)}})\nonumber\\
      & = & \psi_{k-1}^{(n)}(\lambda_1,\ldots,\lambda_{k-2},\mu_{k-1}(n)), \label{local}
   \end{eqnarray}
   where % $\mu_j:=\lambda_j$ for all $j\in\{1,\ldots,k-2\}$ and
   $$
   \mu_{k-1}(n)
   \ :=\ \lambda_{k-1}-n^{1/\alpha_{k-1}}\log(1-(1-e^{-\lambda_k/n^{1/\alpha_k}})^{\alpha_k/\alpha_{k-1}}).
   $$
   A technical but straightforward calculation shows that
   $\mu_{k-1}(n)\to\lambda_{k-1}+\lambda_k^{\alpha_k/\alpha_{k-1}}$ as
   $n\to\infty$. Moreover, by induction, $\psi_{k-1}^{(n)}$ converges
   pointwise to $\psi_{k-1}$ as $n\to\infty$. It is well known that
   the convergence $\psi_{k-1}^{(n)}\to\psi_{k-1}$ of Laplace transforms holds
   even uniformly on any compact subset of $[0,\infty)^{k-1}$.
   Taking these facts into account it follows from (\ref{local}) that
   \begin{eqnarray*}
      \lim_{n\to\infty}\psi_k^{(n)}(\lambda_1,\ldots,\lambda_k)
      & = & \lim_{n\to\infty}\psi_{k-1}^{(n)}(\lambda_1,\ldots,\lambda_{k-2},\mu_{k-1}(n))\\
      & = & \psi_{k-1}(\lambda_1,\ldots,\lambda_{k-2},\lambda_{k-1}+\lambda_k^{\alpha_k/\alpha_{k-1}})
      \ = \ \psi_k(\lambda_1,\ldots,\lambda_k),
   \end{eqnarray*}
   where the last equality holds by Lemma \ref{lap}. The induction is complete.

   The pointwise convergence $\psi_k^{(n)}\to\psi_k$ of the Laplace transforms
   implies the convergence $(Y_{t_1}^{(n)},\ldots,Y_{t_k}^{(n)})\to
   (Y_{t_1},\ldots,Y_{t_k})$ in distribution as $n\to\infty$.

   \vspace{2mm}

   {\bf Part 2.} (Convergence in $D_E[0,\infty)$)
   Recall that $E:=[0,\infty)$ is the state space of the limiting process
   $Y$.
   For $n\in\nz$ and $t\ge 0$ define $E_{n,t}:=\{j/n^{e^t}\,:\,j=n,n+1,\ldots\}$.
   Note that the processes $Y^{(n)}$ are time-inhomogeneous. In order to
   obtain time-homogeneous processes let
   $Z^{(n)}:=(Y_t^{(n)},t)_{t\ge 0}$ and $Z:=(Y_t,t)_{t\ge 0}$ denote
   the space-time processes of $(Y_t^{(n)})_{t\ge 0}$ and $(Y_t)_{t\ge 0}$
   respectively. Note that $Z^{(n)}$ has state space
   $S_n:=\{(j/n^{e^t},t)\,:\,j=n,n+1,\ldots,t\ge 0\}
   =\bigcup_{t\ge 0}(E_{n,t}\times\{t\})$ and that $Z$ has state space
   $S:=E\times [0,\infty)=[0,\infty)^2$. According to Revuz and Yor
   \cite[p.~85, Exercise (1.10)]{revuzyor}, the processes $Z^{(n)}$ and
   $Z$ are time-homogeneous. Define $\pi_n:B(S)\to B(S_n)$ via
   $\pi_ng(y,s):=g(y,s)$ for all $g\in B(S)$ and $(y,s)\in S_n$.
   In the following it is shown that $Z^{(n)}$ converges in $D_S[0,\infty)$
   to $Z$ as $n\to\infty$.
   For $\lambda,\mu>0$
   define the test function $g_{\lambda,\mu}\in \widehat{C}(S)$
   via $g_{\lambda,\mu}(y,s):=e^{-\lambda y-\mu s}$, $(y,s)\in S$.
   By Proposition \ref{laplaceprop}
   it suffices to verify that for every $t\ge 0$
   and $\lambda,\mu>0$,
   \begin{equation} \label{conv}
      \lim_{n\to\infty}\sup_{s\ge 0}\sup_{y\in E_{n,s}}
      |U_t^{(n)}\pi_ng_{\lambda,\mu}(y,s)-\pi_nU_tg_{\lambda,\mu}(y,s)|
      \ =\ 0,
   \end{equation}
   where $U_t^{(n)}:B(S_n)\to B(S_n)$ is defined via
   $U_t^{(n)}g(y,s):=\me(g(Y_{s+t}^{(n)},s+t)\,|\,Y_s^{(n)}=y)$,
   $g\in B(S_n)$, $s\ge 0$, $y\in E_{n,s}$. Note that
   $(U_t^{(n)})_{t\ge 0}$ is the semigroup of $Z^{(n)}$.

   Fix $t\ge 0$ and $\lambda,\mu>0$.
   As before define $\alpha:=e^{-t}$ and $\beta:=e^{-s}$.
   For all $n\in\nz$, $s\ge 0$ and $y\in E_{n,s}$,
   \begin{eqnarray*}
      U_t^{(n)}\pi_ng_{\lambda,\mu}(y,s)
      & = & \me(\pi_ng_{\lambda,\mu}(Y_{s+t}^{(n)},s+t)\,|\,Y_s^{(n)}=y)\\
      & = & \me(\exp(-\lambda Y_{s+t}^{(n)}-\mu(s+t))\,|\,Y_s^{(n)}=y)\\
      & = & (\alpha\beta)^\mu\me(\exp(-\lambda/n^{1/(\alpha\beta)} L_{s+t}^{(n)})\,|\,L_s^{(n)}=yn^{1/\beta})\\
      & = & (\alpha\beta)^\mu\me(\exp(-\lambda/n^{1/(\alpha\beta)}L_t^{(yn^{1/\beta})}))\\
      & = & (\alpha\beta)^\mu(1-(1-e^{-\lambda/n^{1/(\alpha\beta)}})^\alpha)^{yn^{1/\beta}}
   \end{eqnarray*}
   and $\pi_nU_tg_{\lambda,\mu}(y,s)=U_tg_{\lambda,\mu}(y,s)
   =\me(\exp(-\lambda Y_{s+t}-\mu(s+t))\,|\,Y_s=y)
   =(\alpha\beta)^\mu e^{-y\lambda^\alpha}$.
   Define $m:=yn^{1/\beta}\in\{n,n+1,\ldots\}$ and $x:=\lambda/n^{1/(\alpha\beta)}$.
   In the following it is assumed that $n\ge\lambda$ which implies that
   $x\le 1$.
   Using the inequality $|a^m-b^m|\le mr^{m-1}|a-b|$, $m\in\nz$, where
   $r:=\max(|a|,|b|)$,
   it follows that
   \begin{eqnarray*}
   d
   & := & |(1-(1-e^{-\lambda/n^{1/(\alpha\beta)}})^\alpha)^{yn^{1/\beta}} -  e^{-y\lambda^\alpha}|\\
   & = & |(1-(1-e^{-x})^\alpha)^m-(e^{-x^\alpha})^m|\\
   & \le & mr^{m-1} |1-(1-e^{-x})^\alpha-e^{-x^\alpha}|,
   \end{eqnarray*}
   where $r:=\max(1-(1-e^{-x})^\alpha,e^{-x^\alpha})=e^{-x^\alpha}$ by
   Lemma \ref{inequality}. Note that $r\in (0,1)$.

   The map $z\mapsto zr^{z-1}$, $z\ge 0$ takes its
   maximum at the point $z=1/(-\log r)=1/x^\alpha$. Thus,
   $mr^{m-1}\le 1/x^\alpha r^{1/x^\alpha-1}\le 1/x^\alpha$, since
   $r\le 1$ and $x\le 1$, i.e. $1/x^\alpha-1\ge 0$.
   Furthermore, $|1-(1-e^{-x})^\alpha-e^{-x^\alpha}|=
   e^{-x^\alpha}-1+(1-e^{-x})^\alpha
   \le e^{-x^\alpha}-1+x^\alpha\le (x^\alpha)^2/2$.
   Therefore, we obtain the upper bound
   $$
   d\ \le\ \frac{1}{x^\alpha}\frac{(x^\alpha)^2}{2}
   \ =\ \frac{x^\alpha}{2}=\frac{\lambda^\alpha}{2n^{e^s}}
   \ \le \frac{\lambda^\alpha}{2n}.
   $$
   Note that this upper bound does not depend on $y$ and $s$.
   Thus, for all $t\ge 0$, $\lambda,\mu>0$ and all $n\in\nz$ with
   $n\ge\lambda$,
   \begin{eqnarray*}
      &   & \hspace{-15mm}
      \sup_{s\ge 0}
      \sup_{y\in E_{n,s}}
      |U_t^{(n)}\pi_n g_{\lambda,\mu}(y,s)-\pi_nU_tg_{\lambda,\mu}(y,s)|\\
      & = & \sup_{s\ge 0}\sup_{y\in E_{n,s}}
            \underbrace{|e^{-\mu(s+t)}|}_{\le 1}
            \,|(1-(1-e^{-\lambda/n^{e^{s+t}}})^{e^{-t}})^{yn^{e^s}}-e^{-y\lambda^{e^{-t}}}|
      \ \le \ \frac{\lambda^\alpha}{2n}\ \to\ 0
            \label{explicit}
   \end{eqnarray*}
   as $n\to\infty$.
   Therefore, (\ref{conv}) holds for all $t\ge 0$ and all $\lambda,\mu>0$.
   \hfill$\Box$
\end{proof}

\subsection{Appendix}
\setcounter{theorem}{0}
\begin{lemma} \label{inequality}
   For all $x\ge 0$ and all $\alpha\in [0,1]$ we have
   $(1-e^{-x})^\alpha\ge 1-e^{-x^\alpha}$.
\end{lemma}
\begin{proof}
   Fix $\alpha\in [0,1]$. If $x\ge 1$ then $x^\alpha\le x$ and, hence,
   $(1-e^{-x})^\alpha\ge 1-e^{-x}\ge 1-e^{-x^\alpha}$. Assume now that
   $x\in [0,1]$. Then $x^\alpha\ge x$. The function $f(x):=
   (1-e^{-x})^\alpha-1+e^{-x^\alpha}$ satisfies $f(0)=0$ and has derivative
   $f'(x)=\alpha e^{-x}(1-e^{-x})^{\alpha-1}-\alpha x^{\alpha-1}e^{-x^\alpha}$,
   which is nonnegative on $[0,1]$, since $e^{-x}\ge e^{-x^\alpha}$ and
   $(1-e^{-x})^{\alpha-1}\ge x^{\alpha-1}$ for $x\in [0,1]$. From $f(0)=0$
   and $f'(x)\ge 0$ for $x\in [0,1]$ it follows that $f(x)\ge 0$ for
   $x\in [0,1]$, which is the desired inequality.\hfill$\Box$
\end{proof}
\begin{lemma}[Spectral decomposition of $\mathbf\Gamma$ for the Kingman coalescent]
   \label{kingmanspectral}
   \ \\
   The generator $\Gamma=(\gamma_{ij})_{i,j\in\nz}$ of the fixation line
   $(L_t)_{t\ge 0}$ of the Kingman coalescent has spectral decomposition
   $\Gamma=RDL$, where $D=(d_{ij})_{i,j\in\nz}$ is the diagonal matrix
   with entries $d_{ij}=-i(i+1)/2$ for $i=j$ and $d_{ij}=0$ for $i\neq j$,
   and $R=(r_{ij})_{i,j\in\nz}$ and $L=(l_{ij})_{i,j\in\nz}$ are upper
   right triangular matrices with entries
   \begin{equation} \label{rkingman}
      r_{ij}\ =\ (-1)^{j-i}\frac{j!\,(j-1)!\,(i+j)!}{(j-i)!\,i!\,(i-1)!\,(2j)!},
      \qquad i,j\in\nz, i\le j,
   \end{equation}
   and
   \begin{equation} \label{lkingman}
      l_{ij}\ =\ \frac{j!\,(j-1)!\,(2i+1)!}{i!\,(i-1)!\,(j-i)!\,(i+j+1)!},
      \qquad i,j\in\nz, i\le j.
   \end{equation}
\end{lemma}
\begin{remark}
   Note that $l_i(z):=\sum_{j=i}^\infty l_{ij}z^{j+1}$ satisfies the
   differential equation $z^2(1-z)l_i''(z)=i(i+1)l_i(z)$, $i\in\nz$,
   $|z|<1$.
\end{remark}
\begin{proof}
   For a pure birth process the recursion (\ref{rrec}) reduces to $r_{ij}=
   \gamma_i/(\gamma_i-\gamma_j)r_{i+1,j}$, $i\in\{j-1,j-2,\ldots,1\}$, with
   solution $r_{ij}=\prod_{k=i}^{j-1}\gamma_k/(\gamma_k-\gamma_j)$, $i\le j$.
   Thus, for the Kingman coalescent, for all $i,j\in\nz$ with $i\le j$,
   $$
   r_{ij}
   \ =\ \prod_{k=i}^{j-1}\frac{k(k+1)}{k(k+1)-j(j+1)}
   \ =\ \prod_{k=i}^{j-1}\frac{k(k+1)}{(k-j)(k+j+1)}
   \ =\ (-1)^{j-i}\frac{j!\,(j-1)!\,(i+j)!}{(j-i)!\,i!\,(i-1)!\,(2j)!}.
   $$
   Similarly, the recursion (\ref{lrec}) reduces to
   $l_{ij}=\gamma_{j-1}/(\gamma_j-\gamma_i)l_{i,j-1}$, $j\in \{i+1,i+2,\ldots\}$,
   with solution $l_{ij}=\prod_{k=i+1}^j\gamma_{k-1}/(\gamma_k-\gamma_i)$,
   $i\le j$. Thus, for the Kingman coalescent, for all $i,j\in\nz$ with
   $i\le j$,
   $$
   l_{ij}
   \ =\ \prod_{k=i+1}^j\frac{k(k-1)}{k(k+1)-i(i+1)}
   \ =\ \prod_{k=i+1}^{j}\frac{k(k-1)}{(k-i)(k+i+1)}
   \ =\ \frac{j!\,(j-1)!\,(2i+1)!}{i!\,(i-1)!\,(j-i)!\,(i+j+1)!}.
   $$
   \hfill$\Box$
\end{proof}
Let $E$ be locally compact, i.e. every point $x\in E$ has a compact
neighborhood. A function $f:E\to\rz$ vanishes at infinity, if for every
$\varepsilon>0$ there exists a compact $K\subseteq E$ such that
$|f(x)|<\varepsilon$ for all $x\in E\setminus K$. In other words
$\{x\in E\,:\,|f(x)|\ge\varepsilon\}$ is compact.
%A real-valued function $f$ on a locally compact space $E$ is said
%to vanish at infinity if for every $\varepsilon>0$ the set
%$\{x\in E\,:\,|f(x)|\ge\varepsilon\}$ is compact.
In the following $\widehat{C}(E)$ denotes the set of all real-valued
continuous functions on $E$ vanishing at infinity.
\begin{lemma} \label{dense}
   Let $d\in\nz$. The set $D$ of all functions
   $g:[0,\infty)^d\to\rz$ of the form
   $g(y)=\sum_{i_1,\ldots,i_d=1}^m a_{i_1,\ldots,i_d} e^{-(i_1y_1+\cdots+i_dy_d)}$
   with $m\in\nz$ and $a_{i_1,\ldots,i_d}\in\rz$
   is dense in $\widehat{C}([0,\infty)^d)$.
\end{lemma}
\begin{proof}
   Let $g\in\widehat{C}([0,\infty)^d)$. Define $f:[0,1]^d\to\rz$ via
   $f(x):=g(-\log x_1,\ldots,-\log x_d)$ for $x\in (0,1]^d$ and $f(x):=0$
   if $x_j=0$ for some $j\in\{1,\ldots,d\}$.
   Since $g$ is continuous and vanishes at infinity it follows
   that $f$ is continuous. For $n\in\nz$ and $x=(x_1,\ldots,x_d)\in [0,1]^d$ let
   $$
   p_n(x)\ :=\ \sum_{k_1,\ldots,k_d=1}^n
   f\Big(\frac{k_1}{n},\ldots,\frac{k_d}{n}\Big)\prod_{j=1}^d {n\choose k_j}x_j^{k_j}(1-x_j)^{n-k_j}.
   $$
   denote the $n$th multivariate Bernstein polynomial of $f$.
   Note that the sum runs only over $k=(k_1,\ldots,k_d)\in\{1,\ldots,n\}^d$
   (not as usual over $k\in\{0,\ldots,n\}^d$) since $f(x)=0$ if $x_j=0$ for some $j\in\{1,\ldots,d\}$.
   By a $d$-dimensional version of Bernstein's approximation theorem
   (see, for example, \cite[Theorem 8]{duchon}),
   $p_n\to f$ as $n\to\infty$ uniformly on $[0,1]^d$.
   Replacing $x_j$ by $e^{-y_j}$ it follows that $g_n\to g$ as $n\to\infty$
   uniformly on $[0,\infty)^d$, where
   $g_n(y):=p_n(e^{-y_1},\ldots,e^{-y_d})$.
   It remains to note that $g_n\in D$.\hfill$\Box$
\end{proof}
\begin{proposition}[Convergence of Markov processes] %, Laplace version]
\label{laplaceprop}
   Let $d\in\nz$, $E:=[0,\infty)^d$ and $X=(X_t)_{t\ge 0}$ be an $E$-valued
   time-homogeneous Markov process. Furthermore, for
   every $n\in\nz$ let $X^{(n)}=(X_t^{(n)})_{t\ge 0}$ be an
   $E_n$-valued time-homogeneous Markov process with state space $E_n\subseteq E$.
   Let $(T_t)_{t\ge 0}$ and $(T_t^{(n)})_{t\ge 0}$ denote the corresponding
   semigroups. Define $\pi_n:B(E)\to B(E_n)$ via $\pi_nf(x):=f(x)$ for all
   $f\in B(E)$ and $x\in E_n$.
   If, for every $t\ge 0$ and $\lambda\in\nz^d$,
   $$
   \lim_{n\to\infty}\|T_t^{(n)}\pi_nf_\lambda-\pi_nT_tf_\lambda\|
   \ :=\ \lim_{n\to\infty}\sup_{x\in E_n}|T_t^{(n)}\pi_nf_\lambda(x)-\pi_nT_tf_\lambda(x)|
   \ =\ 0,
   $$
   where $f_\lambda(x):=e^{-\langle \lambda,x\rangle}
   :=e^{-(\lambda_1x_1+\cdots+\lambda_dx_d)}$ for all $\lambda\in\nz^d$ and
   $x\in E$, then $X^{(n)}$ converges in $D_E[0,\infty)$ to $X$ as $n\to\infty$.
\end{proposition}
\begin{proof}
   By assumption, $\lim_{n\to\infty}\|T_t^{(n)}\pi_nf-\pi_nT_tf\|=0$
   for all $f\in D$, where $D:=\{f:E\to\rz\,:\,f(x)=\sum_{i=1}^m
   a_i e^{-\langle\lambda,x\rangle},m\in\nz,\lambda\in\nz^d,a_i\in\rz\}$
   Let $f\in\widehat{C}(E)$ and fix $\varepsilon>0$.
   Since $D$ is dense in $\widehat{C}(E)$ by Lemma \ref{dense} there
   exists $h\in D$ such that $\|f-h\|<\varepsilon$. It follows that
   \begin{eqnarray*}
      \|T_t^{(n)}\pi_nf-\pi_nT_tf\|
      & \le & \|T_t^{(n)}\pi_n(f-h)\| + \|T_t^{(n)}\pi_nh-\pi_nT_th\|
              + \|\pi_nT_t(h-f)\|\\
      & \le & \|T_t^{(n)}\|\,\|f-h\| + \|T_t^{(n)}\pi_nh-\pi_nT_th\|
              + \|T_t\|\,\|h-f\|\\
      & \le & 2\varepsilon + \|T_t^{(n)}\pi_nh-\pi_nT_th\|
      \ \to \ 2\varepsilon,\qquad n\to\infty.
   \end{eqnarray*}
   Since $\varepsilon>0$ can be chosen arbitrarily we conclude that
   $\lim_{n\to\infty}\|T_t^{(n)}\pi_nf-\pi_nT_tf\|=0$ for all $f\in\widehat{C}(E)$.
   The result follows from \cite[p.~172, Theorem 2.11]{ethierkurtz}.\hfill$\Box$
\end{proof}
%
%\begin{acknowledgement}
%   ...
%\end{acknowledgement}
%

\end{document}